\documentclass[a4paper, 12pt]{amsart}

\usepackage{xcolor}

\usepackage[english]{babel}

\usepackage{verbatim} 

\usepackage{amsthm}

\newtheorem{theorem}{Theorem}[section]  
\newtheorem{lemma}[theorem]{Lemma}
\newtheorem{corollary}[theorem]{Corollary}
\newtheorem{proposition}[theorem]{Proposition}

\theoremstyle{definition}
\newtheorem{definition}[theorem]{Definition}
\newtheorem{remark}[theorem]{Remark}
\newtheorem{example}[theorem]{Example}
\newtheorem{notation}[theorem]{Notation}
\newtheorem{fact}[theorem]{Fact}

\usepackage{enumerate}

\usepackage[all,cmtip]{xy}

\usepackage{amsmath}
\usepackage{amssymb}
\usepackage{amsfonts}
\usepackage{mathtools}

\usepackage{textcomp}
\usepackage{stmaryrd}
\usepackage[mathscr]{euscript}
\usepackage{bm} 
\usepackage{yfonts} 
\usepackage{blkarray} 
\usepackage{tikz} 
\usetikzlibrary{shapes.misc} 
\tikzset{cross/.style={cross out, draw=black, minimum size=2*(#1-\pgflinewidth), inner sep=0pt, outer sep=0pt}, cross/.default={1pt}}
\usepackage{cancel} 
\usepackage{nicematrix} 

\usepackage{etoolbox}
\let\bbordermatrix\bordermatrix
\patchcmd{\bbordermatrix}{8.75}{4.75}{}{}
\patchcmd{\bbordermatrix}{\left(}{\left[}{}{}
\patchcmd{\bbordermatrix}{\right)}{\right]}{}{}

\usepackage{epsfig}
\usepackage{graphicx}
\usepackage{subfigure}

\usepackage{fancyhdr}

\usepackage[utf8]{inputenc}
\usepackage[T1]{fontenc}

\usepackage[pdftex]{hyperref}

\usepackage{array}

\DeclareGraphicsExtensions{gif,eps,png,jpeg}
\graphicspath{{./eps/}}

\newlength{\defbaselineskip}
\newcommand{\setlinespacing}[1]%
{\setlength{\baselineskip}{#1 \defbaselineskip}}

\setlinespacing{1.5}

\usepackage{indentfirst}
\newcommand{\NN}{\normalfont\mathbb{N}}
\newcommand{\ZZ}{\normalfont\mathbb{Z}}
\newcommand{\QQ}{\normalfont\mathbb{Q}}

\newcommand{\FF}{\normalfont\mathfrak{F}}

\newcommand{\lk}{\textup{lk}}
\newcommand{\cost}{\textup{cost}}
\newcommand{\dm}{\textup{dim}}
\newcommand{\Hr}{\widetilde{H}}

\title{A level initial ideal of the $2$-minors determinantal ideal}

\author[F. Bisio]{Francesco Bisio}
\address[F. Bisio]{Dipartimento di Matematica, Università di Genova, Italy}
\email[F. Bisio]{francesco.bisio@edu.unige.it}

\date{\today}

\begin{document}
	
\pagenumbering{arabic}

\maketitle

\begin{abstract}
For $\Bbbk$ a field, let $X$ a $m \times n$ matrix of variables and $S=\Bbbk[X].$ We consider the determinantal ideal $I_2 \subseteq S$ generated by the $2$-minors of $X.$ In this paper we find a suitable monomial order over $S$ such that $I,$ the initial ideal of $I_2$ with respect to that order, is level, namely, it is Cohen-Macaulay and the socle of an Artinian reduction of the $\NN$-graded algebra $S/I$ is concentrated in only one degree. Moreover, we compare the Betti tables of $I_2$ with the tables of its initial ideals. In the last section, we prove the shellability of the simplicial complex naturally associated to $S/I$ in the case $m<n.$
\end{abstract}

\section{Introduction}

Let $J$ be an ideal in a polynomial ring $T$ with coefficients in a field. Often, to prove that $J$ has a certain property, it is useful to introduce a monomial order $<$ over $T$ and test such property on the monomial ideal $\textup{in}_<(J)$. Indeed the following properties are passed from $T/\textup{in}_<(J)$ to $T/J$: being Cohen-Macaulay, level or Gorenstein. It is very natural and interesting to ask if some sort of converse holds, perhaps if we add some hypothesis on the monomial order $<$.\\

Thanks to the work of Conca and Varbaro (\cite{CoVa}) we know that, if $\textup{in}_<(J)$ is square-free, then $T/J$ is Cohen-Macaulay if and only if $T/\textup{in}_<(J)$ is Cohen-Macaulay. The same statement does not hold if we replace Cohen-Macaulay with Gorenstein, so many authors (\cite{At}, \cite{BruRo}, \cite{CoDenWe}, \cite{CoHosTh}, \cite{JoWe}, \cite{PetPySp} \cite{ReiWe}, \cite{SaStWe}, \cite{SoWe}) have studied the following weaker problem: if a Gorenstein ideal $J$ admits a monomial order such that the initial ideal (with respect to that order) is square-free, can we say that there exists a monomial order $<$ such that $\textup{in}_<(J)$ is square-free and Gorenstein?\\
By Conca and Varbaro we know that if $\textup{in}_<(J)$ is square-free then it is Cohen-Macaulay, but in general it is not Gorenstein: however, if the simplicial complex associated with $\textup{in}_<(J)$ by the Stanley-Reisner correspondence is a join between an $r$-simplex (where $r=-a(T/J)-1$) and a sphere, then $\textup{in}_<(J)$ is guaranteed to be Gorenstein (see \cite[Lemma 2.2]{CoHosTh}).\\

In this paper we study the \textbf{level} property. We can ask if the same philosophy holds for level ideals, i.e. if a level ideal $J$ admits a monomial order such that the initial ideal (with respect to that order) is square-free, then there exists a monomial order $<$ such that $\textup{in}_<(J)$ is square-free and level.\\
Studying this problem, one can try to prove (as in the Gorenstein case) that the simplicial complex associated with $\textup{in}_<(J)$ is a join between a simplex and either a ball or a sphere. However, two complications which are not present in the Gorenstein case arise in the level case.
\begin{itemize}
\item There are level Stanley-Reisner rings $R$ such that $-a(R)$ is not the number of the cone points of the simplicial complex associated with $R.$
\item We know that all simplicial spheres are Gorenstein (and therefore level), but not all simplicial balls are level (see Figure \ref{non-level ball}).
\end{itemize}

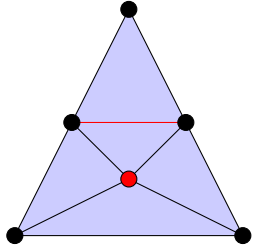
\begin{figure}[h]
\begin{center}
\begin{tikzpicture}[scale=1.5]
 [black,left,red] at (0,0) {};
\node [black,below left,red] at (2,0) {};
\node [black,below right,red] at (1,2) {};
\node [black,right,red] at (1,.5) {};
\node [black,above right,red] at (.5,1) {};
\node [black,above left,red] at (1.5,1) {};
\draw [fill=blue!20] (0,0)--(2,0)--(1,2)--cycle;
\draw (0,0)--(1,.5)--(1.5,1);
\draw [red] (1.5,1)--(.5,1);
\draw (.5,1)--(1,.5)--(2,0);
\draw [fill=black] (0,0) circle (2pt);
\draw [fill=black] (2,0) circle (2pt);
\draw [fill=black] (1,2) circle (2pt);
\draw [fill=red] (1,.5) circle (2pt);
\draw [fill=black] (.5,1) circle (2pt);
\draw [fill=black] (1.5,1) circle (2pt);
\end{tikzpicture}
\caption{A simplicial ball which is not level. As we will see, the minimal generators of the canonical module correspond to the minimal faces that are not contained in the boundary of the ball}
\label{non-level ball}
\end{center}
\end{figure}

Despite these difficulties, in this paper we solve the problem when $J$ is the ideal generated by the $2$-minors of a generic matrix of variables. Indeed we find a monomial order such that the initial ideal of $J$ with respect to that order is square-free and level.\\

Let $n,m \in \NN$, $n \geq m \geq 2.$ We consider the matrix of variables $X \coloneqq \{x_{i,j}\}_{i \in [m]}^{j \in [n]}$, where $[n] \coloneqq \{1,\dots,n\}$, and $S \coloneqq \Bbbk[X]$, with $\Bbbk$ any field. Let $I_2 \coloneqq I_2(X) \subseteq S$ be the ideal generated by the $2$-minors of $X$.\\

In Section \ref{to and simplicial complex} we introduce a specific degrevlex monomial order $<$ over $S$ and the object that we study is the $\NN$-graded $\Bbbk$-algebra $S/I,$ where $I$ is the initial ideal of $I_2$ with respect to $<$. Since $I$ is a square-free monomial ideal of $S,$ we can use the Stanley-Reisner correspondence and consider the simplicial complexes $\Delta_I$ associated with $I$: using this correspondence, we can investigate the algebraic properties of $S/I$ studying the combinatorial properties of $\Delta_I$. The simplicial complex that we actually analyze is $\Delta$, which is the one obtained from $\Delta_I$ after removing its cone points. At the end of the section we show (Proposition \ref{Delta quasimanifold}) that $\Delta$ is a quasimanifold. Moreover, $\Delta$ is a sphere if $m=n$ and a ball if $m<n$.\\

Section \ref{levelness} is the central part of this paper, where we prove that $S/I$ is level, i.e. it is Cohen-Macaulay and its canonical module is generated (as a $\NN$-graded $\Bbbk$-vector space) in a single degree. In order to do that, we use two results of Gr\"abe (\cite{Gr84a} and \cite{Gr84b}) about the canonical module of a Stanley-Reisner ring associated with a quasimanifold. The key is the choice of the proper monomial order $<$. There are several orders such that the initial ideal of $I_2$ is associated to a simplicial complex which, after removing the cone points, turns out to be a ball if $m<n$ and a sphere if $m=n$ (see \cite{CoHosTh}). However, as we have already said, there are balls that are not level.\\
At the end of the section, we study the Betti table of $I_2$ and we compare it with the Betti tables of its initial ideals. We find out that, for $m \geq 3$, $n \geq 4$ and every $\Bbbk$, there is no initial ideal with the same Betti table as $I_2$.\\

In Section \ref{shellablility} we prove the shellability of $\Delta.$ We use methods which are similar to \cite[Section 4]{CoHosTh}.

\section{A Gr\"obner deformation of the determinantal ring} \label{to and simplicial complex}

First of all, let us give the definition of level for a $\NN$-graded Noetherian ring

\begin{definition}
Let $R$ be a Noetherian $\NN$-graded $\Bbbk$-algebra. $R$ is said to be \textbf{level} if these two conditions hold:
\begin{enumerate}[(i)]
\item $R$ is Cohen-Macaulay;
\item the socle of an (equivalently all) homogeneous Artinian reduction of $R$ is concentrated in only one degree.
\end{enumerate}
$I \subseteq R$ an homogeneous ideal is said to be level if $R/I$ is level.
A simplicial complex $\Gamma$ is said to be level (over $\Bbbk$) if its Stanley-Reisner ring $\Bbbk[\Gamma]$ is level. The levelness of $\Gamma$ depends, in general, from the field $\Bbbk$ (so does, for example, the Cohen-Macaulayness).
\end{definition}

\begin{remark}\label{Canonical level}
Notice that a Cohen-Macaulay $\NN$-graded $\Bbbk$-algebra $R$ admits a graded \textbf{canonical module} $\omega_R$ (see \cite[Example 3.6.10 and Proposition 3.6.12]{BruHe93}). The definition above remains equivalent if we require, instead of the second condition, that $\omega_R$ is generated in a single degree as a $\Bbbk$-vector space (see \cite[Proposition 2.1.8(iv)]{GotWa}). More precisely, Goto and Watanabe proved that
$$\omega_R \otimes_R \Bbbk \cong \Bbbk (d_1) \oplus \dots \oplus \Bbbk(d_r)$$
if and only if
$$\textup{Ext}_R^d(\Bbbk,R) \cong \Bbbk(-d_r) \oplus \dots \oplus \Bbbk(-d_1),$$
where $d_1 \leq \dots \leq d_r$ are integers and $d= \textup{dim}(R)$. Since, if $\underline{x}=x_1,\dots,x_d$ is a maximal $R$-regular sequence such that $\textup{deg}(x_i)=b_i,$ putting $b=b_1+\dots+b_d$ we have that
\begin{align*}
\textup{Soc}(R/(\underline{x})R) &\cong \textup{Hom}_{R/(\underline{x})R}(\Bbbk,R/(\underline{x})R) \cong \textup{Ext}_R^d(\Bbbk,R)(-b) \cong \\
&\cong \Bbbk(-d_r-b) \oplus \dots \oplus \Bbbk(-d_1-b),
\end{align*}
our assertion follows.\\
The integer $d_r$ is called the \textbf{a-invariant} of $R.$ By the discussion above, $d_r+b$ is the top non-zero degree of $R/(\underline{x})R$ and $(\omega_R \otimes_R \Bbbk)_{-d_r} \cong (R/(\underline{x})R)_{d_r+b}.$\\
This last observation will help us to prove Lemma \ref{Lemma 1}.
\end{remark}

\begin{definition}
Let us set the monomial order over $S,$ which is the degrevlex order $<$ induced by the following order on the variables: $x_{i,j} > x_{k,l}$ if one among these conditions holds
\begin{enumerate}[(i)]
\item $i \neq j,$ $k \neq l$ and $i<k$;
\item $i \neq j,$ $k \neq l,$ $i=k$ and $j<l$;
\item $i \neq j$ and $k=l$;
\item $i=j$, $k=l$ and $i<k$.
\end{enumerate}
\end{definition}

Basically, to go from the highest to the lowest variable we scroll the rows from the left to the right, but we leave in the last place the variables in the main diagonal, ordering them from the top-left to the bottom-right. We see an example to clarify.

\begin{example}[$m=3, n=4$]
$$X=\begin{bmatrix}
	x_{1,1} & x_{1,2} & x_{1,3} & x_{1,4} \\
	x_{2,1} & x_{2,2} & x_{2,3} & x_{2,4} \\
	x_{3,1} & x_{3,2} & x_{3,3} & x_{3,4}
\end{bmatrix} \rightsquigarrow
\begin{bmatrix}
	3 & 12 & 11 & 10 \\
	9 & 2 & 8 & 7 \\
	6 & 5 & 1 & 4
\end{bmatrix}$$
In this case, the variables are ordered as it follows: $$x_{1,2}>x_{1,3}>x_{1,4}>x_{2,1}>x_{2,3}>x_{2,4}>x_{3,1}>x_{3,2}>x_{3,4}>x_{1,1}>x_{2,2}>x_{3,3}.$$
\end{example}

Since $<$ is a degrevlex order, the set of $2$-minors is a Gröbner basis for $I_2$ with respect to $<$ (see \cite[Proposition 5.3.8]{BruCoRaVa}). This fact implies that $I \coloneqq \textup{in}_<(I_2)$ is generated by the leading terms of the minors and that $I$ is square-free, since there is no squared variable in the minors.\\

We want to prove that $S/I$ is level. Using the Stanley-Reisner correspondence, we call $\Delta_I$ the simplicial complex on $V_I \coloneqq \{(i,j) : i \in [m], j \in [n]\}$ associated with $I.$ Recall that $\Delta_I \coloneqq \{\sigma \subseteq V_I : x_\sigma \notin I\},$ where $x_\sigma \coloneqq \prod_{(i,j) \in \sigma} x_{i,j}.$ One can prove that both $S/I_2$ and $S/I$ are Cohen-Macaulay (see \cite[Theorems 2.4.3 and 4.4.5]{BruCoRaVa}).\\

Notice that, in every minor, at least one of the two monomials has no variable of the main diagonal of $X$ and, since those variables are the "worst" in our degrevlex order, none of them can appear in a leading term of some minor, which are the minimal generators of $I.$ All of this means that the vertices corresponding to the main diagonal variables are in every facet of $\Delta_I,$ i.e. they are cone points for $\Delta_I.$ Conversely, no other vertex $(i,j) \in V_I,$ with $i \neq j$ is a cone point: we consider the minor, up to a sign, $x_{i,j}x_{k,i}-x_{i,i}x_{k,j},$ for some $k \neq i,$ and the leading term $x_{i,j}x_{k,i}$ is a minimal generator of $I;$ so $x_{k,i} \notin I,$ but $x_{i,j}x_{k,i} \in I,$ meaning that $\{(i,j),(k,i)\} \notin \Delta_I,$ but $\{(k,i)\} \in \Delta_I,$ so $(i,j)$ is not a cone point. Summing up, the set of the cone points of $\Delta_I$ is $\textup{Cp} \coloneqq \{(k,k) : k \in[m]\}.$\\

Now we take the simplicial complex obtained starting from $\Delta_I$ and deleting its cone points, that is $\Delta \coloneqq {\Delta_I}_{|_V} = \{\sigma \in \Delta_I : \sigma \subseteq V\},$ where $V \coloneqq V_I \smallsetminus \textup{Cp}.$ The final object of our investigation is $\Bbbk[\Delta],$ the Stanley-Reisner ring of $\Delta.$\\

It is well-known that, in a Noetherian graded ring, going modulo a regular sequence preserves the Cohen-Macaulayness (or non-Cohen-Macaulayness) and it follows from the definition that it holds also for the levelness. It is clear that $\underline{s} \coloneqq x_{1,1},\dots,x_{m,m}$ is a $S/I$-regular sequence, so, since $\Bbbk[\Delta]$ is isomorpihic to $S/(I+(\underline{s})),$ we have that $\Bbbk[\Delta]$ is Cohen-Macaulay and its levelness is equivalent to the levelness of $S/I.$\\

\begin{remark}
In the quadratic case $m=n,$ one can prove (see \cite[Proposition 4.1.9]{BruCoRaVa} and \cite[Theorem 4.13]{CoHosTh}) that $S/I_2$ and $S/I$ are Gorenstein. Recall that a Noetherian $\NN$-graded ring is Gorenstein if it is Cohen-Macaulay and the socle of one (equivalently all) of its Artinian reductions has one generator. It is clear that Gorenstein implies level, so, from now on, we focus on the rectangular case, when $m<n.$
\end{remark}

To study $\Bbbk[\Delta],$ we have to understand how the simplicial complex $\Delta$ looks like; in order to do this, we want to describe how the facets of $\Delta$ are made. Let $\FF(\Delta)$ be the set of the facets of $\Delta$ and let $F \in \FF(\Delta);$ all the variables associated to the vertices in $F$ must not form, pairwise, a leading term of a minor and $F$ has to be maximal respect to this property. Notice that two variables of $X$ that are not in the main diagonal form a leading term of a minor if and only if they are in one of the following situations:

\begin{itemize}
\item in the anti-diagonal of a minor;
\item in the diagonal of a minor and the anti-diagonal of that minor has at least a variable corresponding to a cone point.
\end{itemize}

So if we put the vertices of $\Delta$ in a matrix $M$ in the natural way ($(i,j)$ is in the $i$-th row and in the $j$-th column of $M$), every couple of different vertices $(i,j)$ and $(k,l)$ in $F$ (suppose $k \geq i$) must satify one and only one of the following conditions:

\begin{itemize}
\item $i=k;$
\item $j=l;$
\item $i<k,$ $j<l,$ $i \neq l$ and $k \neq j.$
\end{itemize}

Let us fix some notations now.

\begin{notation}
Let $\sigma$ be a face of $\Delta.$ We define $r^\sigma \coloneqq \{i \in [m] : \exists_{j \in [n]} ((i,j) \in \sigma)\}$ and $c^\sigma \coloneqq \{j \in [n] : \exists_{i \in [m]} ((i,j) \in \sigma)\}$ respectively the rows and the columns of the vertices of $\sigma.$ Then, we call $M^{\sigma}$ the smallest submatrix of $M$ that contains $\sigma$, i.e. the submatrix with rows in $r^\sigma$ and columns in $c^\sigma.$
\end{notation}

Getting back to our argument, we have that $r^F \cap c^F = \emptyset$ (by what we said before) and $r^F \cup c^F = [n]$ (by the maximality of $F$ in $\Delta$). Again, the maximality of $F$ implies that the vertices of $F$ form a path from the top-left vertex of $M^F$ to the bottom-right one and every vertex of the path is one step lower or one step to the right with respect to the previous one. For an example, look at the Figure \ref{Facet path}.

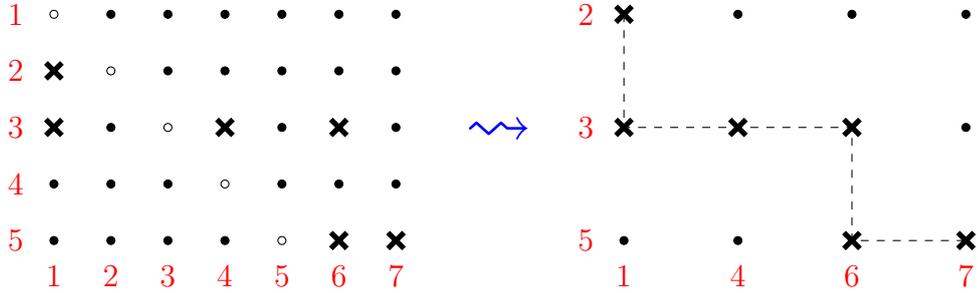
\begin{figure}[h]
\begin{center}
\begin{tikzpicture}[scale=1.5]
\draw [fill=black] (0,0) circle (1pt);
\draw [fill=black] (.5,0) circle (1pt);
\draw [fill=black] (1,0) circle (1pt);
\draw [fill=black] (1.5,0) circle (1pt);
\draw [black] (2,0) circle (1pt);
\draw [fill=black] (0,.5) circle (1pt);
\draw [fill=black] (.5,.5) circle (1pt);
\draw [fill=black] (1,.5) circle (1pt);
\draw [black] (1.5,.5) circle (1pt);
\draw [fill=black] (2,.5) circle (1pt);
\draw [fill=black] (2.5,.5) circle (1pt);
\draw [fill=black] (3,.5) circle (1pt);
\draw [fill=black] (.5,1) circle (1pt);
\draw [black] (1,1) circle (1pt);
\draw [fill=black] (2,1) circle (1pt);
\draw [fill=black] (3,1) circle (1pt);
\draw [black] (.5,1.5) circle (1pt);
\draw [fill=black] (1,1.5) circle (1pt);
\draw [fill=black] (1.5,1.5) circle (1pt);
\draw [fill=black] (2,1.5) circle (1pt);
\draw [fill=black] (2.5,1.5) circle (1pt);
\draw [fill=black] (3,1.5) circle (1pt);
\draw [black] (0,2) circle (1pt);
\draw [fill=black] (.5,2) circle (1pt);
\draw [fill=black] (1,2) circle (1pt);
\draw [fill=black] (1.5,2) circle (1pt);
\draw [fill=black] (2,2) circle (1pt);
\draw [fill=black] (2.5,2) circle (1pt);
\draw [fill=black] (3,2) circle (1pt);
\node [black,above,text=red] at (0,-.5) {$1$};
\node [black,above,text=red] at (.5,-.5) {$2$};
\node [black,above,text=red] at (1,-.5) {$3$};
\node [black,above,text=red] at (1.5,-.5) {$4$};
\node [black,above,text=red] at (2,-.5) {$5$};
\node [black,above,text=red] at (2.5,-.5) {$6$};
\node [black,above,text=red] at (3,-.5) {$7$};
\node [black,right,text=red] at (-.5,0) {$5$};
\node [black,right,text=red] at (-.5,.5) {$4$};
\node [black,right,text=red] at (-.5,1) {$3$};
\node [black,right,text=red] at (-.5,1.5) {$2$};
\node [black,right,text=red] at (-.5,2) {$1$};
\draw (2.5,0) node[cross=5pt,line width=2pt] {};
\draw (3,0) node[cross=5pt,line width=2pt] {};
\draw (0,1) node[cross=5pt,line width=2pt] {};
\draw (1.5,1) node[cross=5pt,line width=2pt] {};
\draw (2.5,1) node[cross=5pt,line width=2pt] {};
\draw (0,1.5) node[cross=5pt,line width=2pt] {};

\node [black,above,text=blue,font=\huge] at (3.9,.75) {$\rightsquigarrow$};

\node [black,above,text=red] at (5,-.5) {$1$};
\node [black,above,text=red] at (6,-.5) {$4$};
\node [black,above,text=red] at (7,-.5) {$6$};
\node [black,above,text=red] at (8,-.5) {$7$};
\node [black,right,text=red] at (4.5,0) {$5$};
\node [black,right,text=red] at (4.5,1) {$3$};
\node [black,right,text=red] at (4.5,2) {$2$};
\draw [dashed] (5,2)--(5,1)--(6,1)--(7,1)--(7,0)--(8,0);
\draw (5,2) node[cross=5pt,line width=2pt] {};
\draw (5,1) node[cross=5pt,line width=2pt] {};
\draw (6,1) node[cross=5pt,line width=2pt] {};
\draw (7,1) node[cross=5pt,line width=2pt] {};
\draw (7,0) node[cross=5pt,line width=2pt] {};
\draw (8,0) node[cross=5pt,line width=2pt] {};
\draw [fill=black] (5,0) circle (1pt);
\draw [fill=black] (6,0) circle (1pt);
\draw [fill=black] (8,1) circle (1pt);
\draw [fill=black] (6,2) circle (1pt);
\draw [fill=black] (7,2) circle (1pt);
\draw [fill=black] (8,2) circle (1pt);
\end{tikzpicture}
\caption{Case $m=5,$ $n=7;$ $r^F=\{2,3,5\},$ $c^F = \{1,4,6,7\}$ and $F=\{(2,1),(3,1),(3,4),(3,6),(5,6),(5,7)\}.$}
\label{Facet path}
\end{center}
\end{figure}

\begin{notation}
Let $F \in \FF(\Delta)$ and consider its vertices in the matrix $M^F.$ A vertex of $F$ is called:
\begin{itemize}
\item a right turning vertex, if it has a previous vertex one step to the left and a following vertex one step lower (e.g. $(3,6)$ in the figure);
\item a left turning vertex, if it has a previous vertex one step higher and a following vertex one step to the right (e.g. $(3,1)$ in the figure);
\item an horizontal vertex, if it is the only vertex in its column (e.g. $(3,4)$ in the figure);
\item a vertical vertex, if it is the only vertex in its row (e.g. $(2,1)$ in the figure).
\end{itemize}
\end{notation}

From this characterization of the facets of $\Delta,$ we can figure out that every $F \in \FF(\Delta)$ has cardinality $|F|=|r^F|+|c^F|-1=n-1,$ therefore $\Delta$ is a pure ($n-2$)-dimensional simplicial complex.\\
Now we want to prove that $\Delta$ is a quasimanifold.

\begin{definition}
A pure $d$-dimensional simplicial complex $\Gamma$ is said to be a $d$-\textbf{quasimanifold} if the following hold:
\begin{enumerate}[(a)]
\item for every ($d-1$)-dimensional face $\sigma \in \Gamma,$ $\lk_\Gamma(\sigma)$ is either one or two points, i.e. $\sigma$ is contained in one or two facets;
\item for every face $\sigma \in \Gamma,$ such that $\dm(\sigma)<d-1,$ $\lk_\Gamma(\sigma)$ is connected, i.e. $\Hr_0(\lk_\Gamma(\sigma);\Bbbk)=0.$
\end{enumerate}
\end{definition}

Before the next proposition, let us give an useful criterion for the Cohen-Macaulayness of simplicial complexes. One can find the proof in \cite[Corollary 5.3.9]{BruHe93}.

\begin{theorem}[Reisner's criterion]
Let $\Gamma$ be a simplicial complex over a field $\Bbbk.$ $\Gamma$ is Cohen-Macaulay if and only if $\Hr_i(\lk_\Gamma(\sigma);\Bbbk)=0,$ for all $\sigma \in \Gamma$ and all $i \neq \dm(\lk_\Gamma(\sigma)).$ 
\end{theorem}

\begin{proposition}\label{Delta quasimanifold}
$\Delta$ is a ($n-2$)-quasimanifold.
\end{proposition}

\begin{proof}
First, note that the property $(b)$ is satisfied, since it is implied by the Cohen-Macaulayness of $\Delta,$ using Reisner's criterion.\\
To consider a general face of $\Delta$ of codimension $1,$ we can remove a general vertex $v=(r_i,c_j)$ from a general facet $F \in \FF(\Delta)$ and analyze all possible cases. Let $F'=F \smallsetminus \{v\}.$\\
If $v$ is a turning vertex of $F$ (see Figure \ref{Left and right vertex}), then $M^F=M^{F'},$ because all the rows and columns are still occupied by $F'.$ The only vertex $v' \neq v$ of $M^F$ we can add to $F'$ such that $F' \cup \{v'\} \in \Delta,$ is the one obtained by flipping $v,$ that is $v'=(r_{i+1},c_{j-1}),$ if $v$ is a right turning vertex, or $v'=(r_{i-1},c_{j+1}),$ if $v$ is a left turning vertex.\\
If $v$ is a vertical vertex (see Figure \ref{Vertical vertex}), then no vertex of $F'$ is in the row $r_i$ and note that no other vertex of $M^F$ can be added to $F'$ to obtain a facet. Since the the row $r_i$ is now free, we have another choice: we can consider the submatrix that comes when we remove from $M^F$ the row $r_i$ and we add again $r_i$ as a column; notice that there is only one suitable vertex $v'$ in the column $r_i$ that we can add to $F'.$\\
If $v$ is an horizontal vertex, then, if $c_j \leq m,$ the process is analogous, with the proper modifications (changing rows with columns), to the previous one; otherwise, we cannot add the row $c_j,$ because the matrix $M$ has only $m$ rows.\\
Notice that in the latter case $F'$ is contained only in one facet, that is $F,$ and in all the other cases is contained in two. This proves the property $(a)$ and the statement.\\
\end{proof}

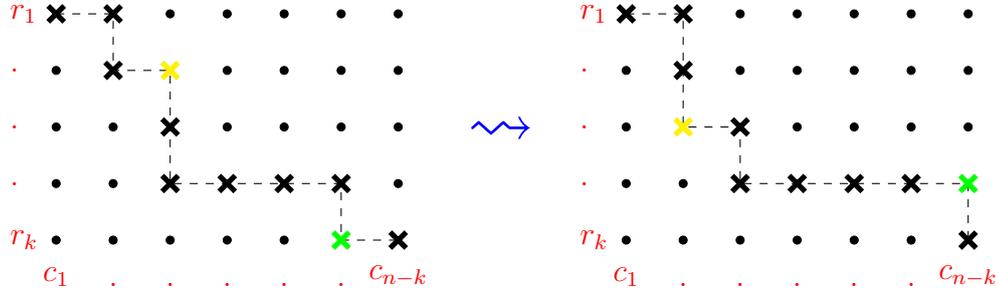
\begin{figure}[h]
\begin{center}
\begin{tikzpicture}[scale=1.5]
\draw [fill=black] (1,2) circle (1pt);
\draw [fill=black] (1.5,2) circle (1pt);
\draw [fill=black] (2,2) circle (1pt);
\draw [fill=black] (2.5,2) circle (1pt);
\draw [fill=black] (3,2) circle (1pt);
\draw [fill=black] (0,1.5) circle (1pt);
\draw [fill=black] (1.5,1.5) circle (1pt);
\draw [fill=black] (2,1.5) circle (1pt);
\draw [fill=black] (2.5,1.5) circle (1pt);
\draw [fill=black] (3,1.5) circle (1pt);
\draw [fill=black] (0,1) circle (1pt);
\draw [fill=black] (.5,1) circle (1pt);
\draw [fill=black] (1.5,1) circle (1pt);
\draw [fill=black] (2,1) circle (1pt);
\draw [fill=black] (2.5,1) circle (1pt);
\draw [fill=black] (3,1) circle (1pt);
\draw [fill=black] (0,.5) circle (1pt);
\draw [fill=black] (.5,.5) circle (1pt);
\draw [fill=black] (3,.5) circle (1pt);
\draw [fill=black] (0,0) circle (1pt);
\draw [fill=black] (.5,0) circle (1pt);
\draw [fill=black] (1,0) circle (1pt);
\draw [fill=black] (1.5,0) circle (1pt);
\draw [fill=black] (2,0) circle (1pt);
\draw [dashed] (0,2)--(.5,2)--(.5,1.5)--(1,1.5)--(1,1)--(1,.5)--(1.5,.5)--(2,.5)--(2.5,.5)--(2.5,0)--(3,0);
\node [black,above,text=red] at (0,-.5) {$c_1$};
\node [black,above,text=red] at (.5,-.5) {$.$};
\node [black,above,text=red] at (1,-.5) {$.$};
\node [black,above,text=red] at (1.5,-.5) {$.$};
\node [black,above,text=red] at (2,-.5) {$.$};
\node [black,above,text=red] at (2.5,-.5) {.};
\node [black,above,text=red] at (3,-.5) {$c_{n-k}$};
\node [black,right,text=red] at (-.5,0) {$r_k$};
\node [black,right,text=red] at (-.5,.5) {$.$};
\node [black,right,text=red] at (-.5,1) {$.$};
\node [black,right,text=red] at (-.5,1.5) {$.$};
\node [black,right,text=red] at (-.5,2) {$r_1$};
\draw (0,2) node[cross=5pt,line width=2pt] {};
\draw (.5,2) node[cross=5pt,line width=2pt] {};
\draw (.5,1.5) node[cross=5pt,line width=2pt] {};
\draw (1,1.5) node[cross=5pt,line width=2pt,yellow] {};
\draw (1,1) node[cross=5pt,line width=2pt] {};
\draw (1,.5) node[cross=5pt,line width=2pt] {};
\draw (1.5,.5) node[cross=5pt,line width=2pt] {};
\draw (2,.5) node[cross=5pt,line width=2pt] {};
\draw (2.5,.5) node[cross=5pt,line width=2pt] {};
\draw (2.5,0) node[cross=5pt,line width=2pt,green] {};
\draw (3,0) node[cross=5pt,line width=2pt] {};

\node [black,above,text=blue,font=\huge] at (3.9,.75) {$\rightsquigarrow$};

\draw [fill=black] (6,2) circle (1pt);
\draw [fill=black] (6.5,2) circle (1pt);
\draw [fill=black] (7,2) circle (1pt);
\draw [fill=black] (7.5,2) circle (1pt);
\draw [fill=black] (8,2) circle (1pt);
\draw [fill=black] (5,1.5) circle (1pt);
\draw [fill=black] (6,1.5) circle (1pt);
\draw [fill=black] (6.5,1.5) circle (1pt);
\draw [fill=black] (7,1.5) circle (1pt);
\draw [fill=black] (7.5,1.5) circle (1pt);
\draw [fill=black] (8,1.5) circle (1pt);
\draw [fill=black] (5,1) circle (1pt);
\draw [fill=black] (6.5,1) circle (1pt);
\draw [fill=black] (7,1) circle (1pt);
\draw [fill=black] (7.5,1) circle (1pt);
\draw [fill=black] (8,1) circle (1pt);
\draw [fill=black] (5,.5) circle (1pt);
\draw [fill=black] (5.5,.5) circle (1pt);
\draw [fill=black] (6,0) circle (1pt);
\draw [fill=black] (5,0) circle (1pt);
\draw [fill=black] (5.5,0) circle (1pt);
\draw [fill=black] (6.5,0) circle (1pt);
\draw [fill=black] (7,0) circle (1pt);
\draw [fill=black] (7.5,0) circle (1pt);
\draw [dashed] (5,2)--(5.5,2)--(5.5,1.5)--(5.5,1)--(6,1)--(6,.5)--(6.5,.5)--(7,.5)--(7.5,.5)--(8,.5)--(8,0);
\node [black,above,text=red] at (5,-.5) {$c_1$};
\node [black,above,text=red] at (5.5,-.5) {$.$};
\node [black,above,text=red] at (6,-.5) {$.$};
\node [black,above,text=red] at (6.5,-.5) {$.$};
\node [black,above,text=red] at (7,-.5) {$.$};
\node [black,above,text=red] at (7.5,-.5) {.};
\node [black,above,text=red] at (8,-.5) {$c_{n-k}$};
\node [black,right,text=red] at (4.5,0) {$r_k$};
\node [black,right,text=red] at (4.5,.5) {$.$};
\node [black,right,text=red] at (4.5,1) {$.$};
\node [black,right,text=red] at (4.5,1.5) {$.$};
\node [black,right,text=red] at (4.5,2) {$r_1$};
\draw (5,2) node[cross=5pt,line width=2pt] {};
\draw (5.5,2) node[cross=5pt,line width=2pt] {};
\draw (5.5,1.5) node[cross=5pt,line width=2pt] {};
\draw (5.5,1) node[cross=5pt,line width=2pt,yellow] {};
\draw (6,1) node[cross=5pt,line width=2pt] {};
\draw (6,.5) node[cross=5pt,line width=2pt] {};
\draw (6.5,.5) node[cross=5pt,line width=2pt] {};
\draw (7,.5) node[cross=5pt,line width=2pt] {};
\draw (7.5,.5) node[cross=5pt,line width=2pt] {};
\draw (8,.5) node[cross=5pt,line width=2pt,green] {};
\draw (8,0) node[cross=5pt,line width=2pt] {};

\end{tikzpicture}
\caption{$v$ is a right (yellow) or left (green) turning vertex of $F.$}
\label{Left and right vertex}
\end{center}
\end{figure}

\begin{figure}[h]
\begin{center}
\begin{tikzpicture}[scale=1.5]
\draw [fill=black] (1,2) circle (1pt);
\draw [fill=black] (1.5,2) circle (1pt);
\draw [fill=black] (2,2) circle (1pt);
\draw [fill=black] (2.5,2) circle (1pt);
\draw [fill=black] (3,2) circle (1pt);
\draw [fill=black] (0,1.5) circle (1pt);
\draw [fill=black] (1.5,1.5) circle (1pt);
\draw [fill=black] (2,1.5) circle (1pt);
\draw [fill=black] (2.5,1.5) circle (1pt);
\draw [fill=black] (3,1.5) circle (1pt);
\draw [fill=black] (0,1) circle (1pt);
\draw [fill=black] (.5,1) circle (1pt);
\draw [fill=black] (1.5,1) circle (1pt);
\draw [fill=black] (2,1) circle (1pt);
\draw [fill=black] (2.5,1) circle (1pt);
\draw [fill=black] (3,1) circle (1pt);
\draw [fill=black] (0,.5) circle (1pt);
\draw [fill=black] (.5,.5) circle (1pt);
\draw [fill=black] (3,.5) circle (1pt);
\draw [fill=black] (0,0) circle (1pt);
\draw [fill=black] (.5,0) circle (1pt);
\draw [fill=black] (1,0) circle (1pt);
\draw [fill=black] (1.5,0) circle (1pt);
\draw [fill=black] (2,0) circle (1pt);
\draw [dashed] (0,2)--(.5,2)--(.5,1.5)--(1,1.5)--(1,1)--(1,.5)--(1.5,.5)--(2,.5)--(2.5,.5)--(2.5,0)--(3,0);
\node [black,above,text=red] at (0,-.5) {$c_1$};
\node [black,above,text=red] at (.5,-.5) {$.$};
\node [black,above,text=red] at (1,-.5) {$.$};
\node [black,above,text=red] at (1.5,-.5) {$.$};
\node [black,above,text=red] at (2,-.5) {$.$};
\node [black,above,text=red] at (2.5,-.5) {.};
\node [black,above,text=red] at (3,-.5) {$c_{n-k}$};
\node [black,right,text=red] at (-.5,0) {$r_k$};
\node [black,right,text=red] at (-.5,.5) {$.$};
\node [black,right,text=red] at (-.5,1) {$r_i$};
\node [black,right,text=red] at (-.5,1.5) {$.$};
\node [black,right,text=red] at (-.5,2) {$r_1$};
\draw (0,2) node[cross=5pt,line width=2pt] {};
\draw (.5,2) node[cross=5pt,line width=2pt] {};
\draw (.5,1.5) node[cross=5pt,line width=2pt] {};
\draw (1,1.5) node[cross=5pt,line width=2pt] {};
\draw (1,1) node[cross=5pt,line width=2pt,green] {};
\draw (1,.5) node[cross=5pt,line width=2pt] {};
\draw (1.5,.5) node[cross=5pt,line width=2pt] {};
\draw (2,.5) node[cross=5pt,line width=2pt] {};
\draw (2.5,.5) node[cross=5pt,line width=2pt] {};
\draw (2.5,0) node[cross=5pt,line width=2pt] {};
\draw (3,0) node[cross=5pt,line width=2pt] {};
			
\node [black,above,text=blue,font=\huge] at (3.9,.75) {$\rightsquigarrow$};

\draw [fill=black] (6,1.5) circle (1pt);
\draw [fill=black] (6.5,1.5) circle (1pt);
\draw [fill=black] (7,1.5) circle (1pt);
\draw [fill=black] (7.5,1.5) circle (1pt);
\draw [fill=black] (8,1.5) circle (1pt);
\draw [fill=black] (8.5,1.5) circle (1pt);
\draw [fill=black] (5,1) circle (1pt);
\draw [fill=black] (6.5,1) circle (1pt);
\draw [fill=black] (7,1) circle (1pt);
\draw [fill=black] (7.5,1) circle (1pt);
\draw [fill=black] (8,1) circle (1pt);
\draw [fill=black] (8.5,1) circle (1pt);
\draw [fill=black] (5,.5) circle (1pt);
\draw [fill=black] (5.5,.5) circle (1pt);
\draw [fill=black] (8.5,.5) circle (1pt);
\draw [fill=black] (5,0) circle (1pt);
\draw [fill=black] (5.5,0) circle (1pt);
\draw [fill=black] (6,0) circle (1pt);
\draw [fill=black] (6.5,0) circle (1pt);
\draw [fill=black] (7,0) circle (1pt);
\draw [fill=black] (7.5,0) circle (1pt);
\draw [dashed] (5,1.5)--(5.5,1.5)--(5.5,1)--(6,1)--(6,.5)--(6.5,.5)--(7,.5)--(7.5,.5)--(8,.5)--(8,0)--(8.5,0);
\node [black,above,text=red] at (5,-.5) {$c_1$};
\node [black,above,text=red] at (5.5,-.5) {$.$};
\node [black,above,text=red] at (6,-.5) {$.$};
\node [black,above,text=red] at (6.5,-.5) {$.$};
\node [black,above,text=red] at (7,-.5) {$.$};
\node [black,above,text=red] at (7.5,-.5) {$r_i$};
\node [black,above,text=red] at (8,-.5) {$.$};
\node [black,above,text=red] at (8.5,-.5) {$c_{n-k}$};
\node [black,right,text=red] at (4.5,0) {$r_k$};
\node [black,right,text=red] at (4.5,.5) {$.$};
\node [black,right,text=red] at (4.5,.75) {$\hat{r_i}$};
\node [black,right,text=red] at (4.5,1) {$.$};
\node [black,right,text=red] at (4.5,1.5) {$r_1$};
\draw (5,1.5) node[cross=5pt,line width=2pt] {};
\draw (5.5,1.5) node[cross=5pt,line width=2pt] {};
\draw (5.5,1) node[cross=5pt,line width=2pt] {};
\draw (6,1) node[cross=5pt,line width=2pt] {};
\draw (6,.5) node[cross=5pt,line width=2pt] {};
\draw (6.5,.5) node[cross=5pt,line width=2pt] {};
\draw (7,.5) node[cross=5pt,line width=2pt] {};
\draw (7.5,.5) node[cross=5pt,line width=2pt,green] {};
\draw (8,.5) node[cross=5pt,line width=2pt] {};
\draw (8,0) node[cross=5pt,line width=2pt] {};
\draw (8.5,0) node[cross=5pt,line width=2pt] {};
			
\end{tikzpicture}
\caption{$v$ is a vertical vertex of $F.$}
\label{Vertical vertex}
\end{center}
\end{figure}
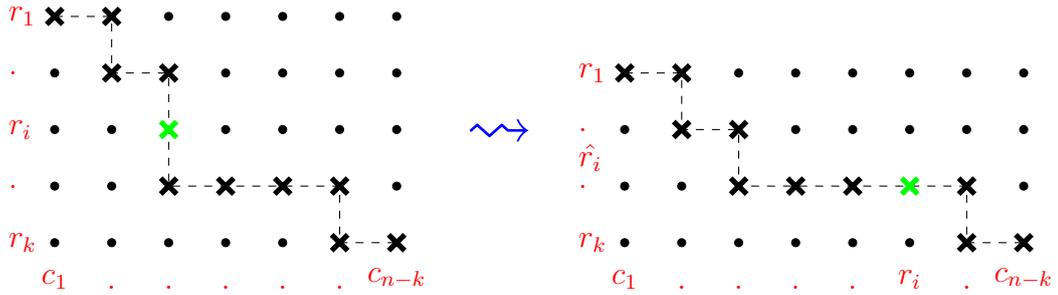

\begin{remark}\label{Boundary faces}
In the latter proof, we have characterized those faces of $\Delta$ of codimension $1$ that are contained in only one facet (all the other ones are contained in two facets). Those faces $\sigma$ are the ($n-3$)-faces that satisfy one of the two following equivalent conditions:
\begin{itemize}
\item $r^\sigma \cup c^\sigma = [n] \smallsetminus \{l\},$ with $l>m;$
\item there exists $F \in \FF(\Delta)$ such that $\sigma$ can be obtained removing from $F$ an horizontal vertex which stays in the column $l,$ with $l>m.$
\end{itemize}
Notice that such faces always exist in the rectangular case $m<n$. We can take the facet $F = \{(1,k) : k \in \{2,\dots,n\}\}$ and remove the vertex $(1,n)$ to obtain the face that we want.\\
If we look at the quadratic case $m=n$, $\Delta$ is still a quasimanifold, but all the faces of codimension 1 are contained in exactly two facets. This difference is essentially the reason why $\Delta$ is Gorenstein in the quadratic case and is not Gorenstein in the rectangular case.
\end{remark}

\section{The canonical module and the levelness} \label{levelness}

Following the idea of the Remark \ref{Canonical level}, which tells us that we can test the levelness of $\Bbbk[\Delta]$ studying his canonical module $\omega_{\Bbbk[\Delta]},$ in this section we investigate about its structure.\\
Let $\Gamma$ be a $d$-dimensional simplicial complex on $[r],$ for some $r \in \NN,$ over a field $\Bbbk.$ For all $\sigma \in \Gamma$ and $U=(U_1,\dots,U_r) \in \NN^{[r]},$ we define the contrastar of $\sigma$ in $\Gamma$ as the simplicial complex $\cost_\Gamma(\sigma) \coloneqq \{\tau \in \Gamma : \sigma \nsubseteq \tau\}$ and the support of $U$ as $s(U) \coloneqq \{i \in [r] : U_i \neq 0\}.$\\
Let $H_i(\Gamma,\Gamma';\Bbbk)$ be the $i$-th relative homology of the pair of simplicial complexes $\Gamma' \subseteq \Gamma$ and $\Hr_i(\Gamma;\Bbbk)$ the $i$-th reduced homology of $\Gamma$. One may check in \cite{Gr84a} that, for all $i \in \ZZ$ and all $\sigma \in \Gamma$, the following isomorphism of $\Bbbk$-vector spaces holds:

\begin{equation}\label{Eq 1}
\Hr_{i-|\sigma|}(\lk_\Gamma(\sigma);\Bbbk) \cong H_i(\Gamma,\cost_\Gamma(\sigma);\Bbbk).
\end{equation}

\begin{remark}
Recall that an element of $H_i(\Gamma,\cost_\Gamma(\sigma);\Bbbk)$ is of the type

$$\overline{\sum_{\substack{\eta \in \Gamma \smallsetminus \cost_\Gamma(\sigma) \\ \dm(\eta)=i}} \hspace{-.2cm} a_\eta \eta} = \sum_{\substack{\eta \in \Gamma \smallsetminus \cost_\Gamma(\sigma) \\ \dm(\eta)=i}} \hspace{-.2cm} a_\eta \eta + \widetilde{C}_i(\cost_\Gamma(\sigma)) \in H_i(\Gamma,\cost_\Gamma(\sigma);\Bbbk),$$

where, for all $\eta,$ $a_\eta \in \Bbbk.$ For $\sigma \subseteq \tau \in \Gamma$ and $i \in \ZZ$ there is a natural restriction map $f_i(\sigma,\tau)$

$$\overline{\sum_{\substack{\eta \in \Gamma \smallsetminus \cost_\Gamma(\sigma) \\ \dm(\eta)=i}} \hspace{-.2cm} a_\eta \eta} \in H_i(\Gamma,\cost_\Gamma(\sigma);\Bbbk) \longmapsto \overline{\sum_{\substack{\eta \in \Gamma \smallsetminus \cost_\Gamma(\tau) \\ \dm(\eta)=i}} \hspace{-.2cm} a_\eta \eta} \in H_i(\Gamma,\cost_\Gamma(\tau);\Bbbk)$$
\end{remark}

The next theorem from Gr\"abe (see \cite[Theorem 4]{Gr84a}) tells us about the structure of the canonical module of a Stanley-Reisner ring.

\begin{theorem}\label{Grabe canonical}
Let $\Gamma$ be a $d$-dimensional simplicial complex on $[r]$ for some $r \in \NN$. There exists an isomorphism of $\NN^r$-graded $\Bbbk[\Gamma]$-modules

$$\omega_{\Bbbk[\Gamma]} \cong \displaystyle \bigoplus_{U \in \NN^r, s(U) \in \Gamma} H_d(\Gamma,\cost_\Gamma(s(U));\Bbbk)$$

The structure of $\Bbbk[\Gamma]$-module, for $U,W \in \NN^r$ such that $s(U) \in \Gamma,$ is given by

$$\cdot x_W = \begin{cases}
	f_d(s(U),s(U+W)) & \text{if } s(U+W) \in \Gamma \\
	0-\text{map} & \text{otherwise}
\end{cases}$$
\end{theorem}

In our case, the simplicial complex $\Delta$ is also a quasimanifold. Another theorem of Gr\"abe (see \cite[Hauptlemma 3.2]{Gr84b}) tells us an interesting fact about the canonical module in this case.

\begin{theorem}\label{Grabe injective}
Let $\Gamma$ be a connected $d$-quasimanifold, then all the maps $f_d$ of the previous theorem are injective. In particular, if $\sigma = s(U)$ and $\tau =s(U+W),$

$$\dm_\Bbbk(H_d(\Gamma,\cost_\Gamma(\sigma);\Bbbk)) \leq \dm_\Bbbk(H_d(\Gamma,\cost_\Gamma(\tau);\Bbbk))$$
\end{theorem}

Notice that our simplicial complex $\Delta$ is a Cohen-Macaulay (which implies connected) quasimanifold, thus it satifies the hypothesis of the latter theorem. Using the isomorphism \ref{Eq 1} we have that, for $\sigma \subseteq \tau \in \Delta,$
$$\dm_\Bbbk(\Hr_{\dm(\lk_\Delta(\sigma))}(\lk_\Delta(\sigma);\Bbbk)) \leq \dm_\Bbbk(\Hr_{\dm(\lk_\Delta(\tau))}(\lk_\Delta(\tau);\Bbbk)),$$

where $\dm(\lk_\Delta(\sigma))=n-2-|\sigma|$ and $\dm(\lk_\Delta(\tau))=n-2-|\tau|$ hold since $\Delta$ is pure.

\begin{notation}
From now on, for $\sigma \in \Delta,$ we set $d_\sigma \coloneqq \dm_\Bbbk(\Hr_{\dm(\lk_\Delta(\sigma))}(\lk_\Delta(\sigma);\Bbbk)).$
\end{notation}

\begin{remark}
Let be $\sigma$ be a facet of $\Delta,$ then $\lk_\Delta(\sigma)=\{\emptyset\}$ and so we have $d_\sigma = \dm_\Bbbk(\Hr_{-1}(\{\emptyset\};\Bbbk)) = 1.$\\
Let $\sigma \in \Delta$ of codimension $1$ (i.e. $\dm(\sigma)=n-3$), then $\lk_\Delta(\sigma)$ is either one or two points and so

$$d_\sigma = \dm_\Bbbk(\Hr_0(\lk_\Delta(\sigma);\Bbbk)) = \begin{cases}
	0 & \text{if } \lk_\Delta(\sigma) \text{ is one point} \\
	1 & \text{if } \lk_\Delta(\sigma) \text{ are two points}
\end{cases}.$$
Since we are studying the rectangular case, by Remark \ref{Boundary faces}, there exists always $\sigma$ a $(n-3)$-face of $\Delta$ such that $d_\sigma = 0,$ then $d_{\{\emptyset\}} \leq d_\sigma = 0.$\\
Now let $\{\emptyset\} = \sigma_0 \subset \sigma_1 \subset \dots \subset \sigma_{n-1}$ be a sequence of faces in $\Delta$ such that $|\sigma_i|=i,$ then we have $0 = d_{\sigma_0} \leq d_{\sigma_1} \leq \dots \leq d_{\sigma_{n-1}} = 1.$ That means that there exists $k \in [n-1]$ such that $d_{\sigma_i}$ is $0$ if $i<k$ and $1$ if $i \geq k.$ If we consider the $\NN$-grading on $\omega_{\Bbbk[\Delta]},$ we have that $H_{n-2}(\Delta,\cost_\Gamma(\sigma_i);\Bbbk),$ with $i \geq k,$ is a one-dimensional $\Bbbk$-vector subspace of $(\omega_{\Bbbk[\Delta]})_i.$ So we can think about such $\sigma_i$ as a generator of $\omega_{\Bbbk[\Delta]}$ in degree $i.$ We would like to show that the generators which are minimal with respect to the inclusion have cardinality $n-m$ and so, by Remark \ref{Canonical level}, the levelness of $\Delta.$
\end{remark}

Now let us prove two lemmas before stating the main theorem.

\begin{lemma}\label{Lemma 1}
The minimum non-zero degree of $\omega_{\Bbbk[\Delta]}$ is $n-m.$ Moreover it holds $\dm_\Bbbk((\omega_{\Bbbk[\Delta]})_{n-m}) = \binom{n-1}{m-1}.$
\end{lemma}

\begin{proof}
Recall that the $a$-invariant of an $\NN$-graded Cohen-Macaulay $\Bbbk$-algebra $R$ is the number $a(R) \coloneqq -\min\{i \in \NN: (\omega_R)_i \neq 0\}.$ It is well known (see \cite[Corollary 1.5]{BruHe92}) that $a(S/I_2)=-n.$ Since $I$ is square-free, by \cite[Corollary 2.3.8]{BruCoRaVa}, $S/I$ is Cohen-Macaulay and $a(S/I)=-n.$ Since dim$(S/I)=\textup{dim}(S/I_2)=n+m-1$, if $\underline{x}$ is a linear system of parameters of $S/I$, it holds
$$\textup{dim}_{\Bbbk}((\omega_{S/I})_n) = \textup{dim}_{\Bbbk}((S/(I+(\underline{x})))_{m-1}) = \binom{n-1}{m-1}.$$
The first equality follows from the final part of Remark \ref{Canonical level} and the second one from \cite[Proposition 4.1.9(c)]{BruCoRaVa}.\\
The conclusion follows because, by \cite[Corollary 3.6.14]{BruHe93}, $(\omega_{\Bbbk[\Delta]})_{n-m} = (\omega_{S/I})_n.$\\
\end{proof}

This lemma tells us that $\omega_{\Bbbk[\Delta]}$ has no generator below degree $n-m,$ however there could be some minimal generator in degree higher than $n-m.$ In the theorem we will show that this is not the case.

\begin{lemma}\label{Lemma 2}
Let \textup{Bd} $\coloneqq \{\sigma \in \Delta : \dm(\sigma)=n-3 \hspace{.1cm} \wedge \hspace{.1cm} H_{n-2}(\Delta,\cost_\Delta(\sigma);\Bbbk) = 0\}$ and $C \coloneqq \{\varepsilon \in \Delta : \nexists_{\sigma \in \textup{Bd}}(\varepsilon \subseteq \sigma) \wedge \forall_{\delta \subsetneq \varepsilon}(\exists_{\sigma \in \textup{Bd}}(\delta \subseteq \sigma))\}.$ Then
$$C=\{\{(i_1,m+1),(i_2,m+2),\dots,(i_{n-m},n)\} : \forall_{j \in [n-m-1]}(1 \leq i_j \leq i_{j+1} \leq m)\}.$$
Moreover, $|C|=\binom{n-1}{m-1}.$
\end{lemma}

\begin{proof}
By Remark \ref{Boundary faces}, it follows that $\sigma \in$ Bd if and only if $\dm(\sigma)=n-3$ and there exists a unique $l \in \{m+1,\dots,n\}$ such that $l \notin c^\sigma.$\\
The definition of $C$ tell us that $\varepsilon \in C$ are the minimal faces of $\Delta$ that are not contained in any $\sigma \in$ Bd. Notice that, by what we just said, what appens in the first $m$ columns of $M$ does not affect the fact that $\varepsilon \subseteq \sigma \in$ Bd or not, then, by the minimality of $\varepsilon \in C,$ we have that all the vertices of $\varepsilon$ are contained in the last $n-m$ columns of $M.$ Now notice that $\varepsilon \in C$ must occupy all these columns and, again by minimality, $\varepsilon$ consists only of $n-m$ vertices, one for each column.\\
Summing up and remembering the conditions under which $\varepsilon$ is a face of $\Delta,$ we have that $\varepsilon = \{(i_1,m+1),\dots,(i_{n-m},n)\},$ for some $1 \leq i_1 \leq \dots \leq i_{n-m} \leq m.$\\
For the second part, it suffices to notice that $C$ is in bijection with the set of all possible non-decreasing sequences of length $n-m$ of numbers in $[m],$ then $|C|=\binom{m+(n-m)-1}{n-m}=\binom{n-1}{n-m}=\binom{n-1}{m-1}.$\\
\end{proof}

\begin{figure}[h]
\begin{center}
\begin{tikzpicture}[scale=1.5]
\draw [fill=black] (0,0) circle (1pt);
\draw [fill=black] (.5,0) circle (1pt);
\draw [fill=black] (1,0) circle (1pt);
\draw [fill=black] (1.5,0) circle (1pt);
\draw [black] (2,0) circle (1pt);
\draw [fill=black] (2.5,0) circle (1pt);
\draw [fill=black] (3,0) circle (1pt);
\draw [fill=black] (3.5,0) circle (1pt);
\draw [fill=black] (0,.5) circle (1pt);
\draw [fill=black] (.5,.5) circle (1pt);
\draw [fill=black] (1,.5) circle (1pt);
\draw [black] (1.5,.5) circle (1pt);
\draw [fill=black] (2,.5) circle (1pt);
\draw [fill=black] (2.5,.5) circle (1pt);
\draw [fill=black] (3,.5) circle (1pt);
\draw [fill=black] (4,.5) circle (1pt);
\draw [fill=black] (0,1) circle (1pt);
\draw [fill=black] (.5,1) circle (1pt);
\draw [black] (1,1) circle (1pt);
\draw [fill=black] (1.5,1) circle (1pt);
\draw [fill=black] (2,1) circle (1pt);
\draw [fill=black] (2.5,1) circle (1pt);
\draw [fill=black] (3,1) circle (1pt);
\draw [fill=black] (3.5,1) circle (1pt);
\draw [fill=black] (4,1) circle (1pt);
\draw [fill=black] (0,1.5) circle (1pt);
\draw [black] (.5,1.5) circle (1pt);
\draw [fill=black] (1,1.5) circle (1pt);
\draw [fill=black] (1.5,1.5) circle (1pt);
\draw [fill=black] (2,1.5) circle (1pt);
\draw [fill=black] (3.5,1.5) circle (1pt);
\draw [fill=black] (4,1.5) circle (1pt);
\draw [black] (0,2) circle (1pt);
\draw [fill=black] (.5,2) circle (1pt);
\draw [fill=black] (1,2) circle (1pt);
\draw [fill=black] (1.5,2) circle (1pt);
\draw [fill=black] (2,2) circle (1pt);
\draw [fill=black] (2.5,2) circle (1pt);
\draw [fill=black] (3,2) circle (1pt);
\draw [fill=black] (3.5,2) circle (1pt);
\draw [fill=black] (4,2) circle (1pt);
\node [black,above,text=red] at (0,-.5) {$1$};
\node [black,above,text=red] at (.5,-.5) {$2$};
\node [black,above,text=red] at (1,-.5) {$3$};
\node [black,above,text=red] at (1.5,-.5) {$4$};
\node [black,above,text=red] at (2,-.5) {$5$};
\node [black,above,text=red] at (2.5,-.5) {$6$};
\node [black,above,text=red] at (3,-.5) {$7$};
\node [black,above,text=red] at (3.5,-.5) {$8$};
\node [black,above,text=red] at (4,-.5) {$9$};
\node [black,right,text=red] at (-.5,0) {$5$};
\node [black,right,text=red] at (-.5,.5) {$4$};
\node [black,right,text=red] at (-.5,1) {$3$};
\node [black,right,text=red] at (-.5,1.5) {$2$};
\node [black,right,text=red] at (-.5,2) {$1$};
\draw (2.5,1.5) node[cross=5pt,line width=2pt] {};
\draw (3,1.5) node[cross=5pt,line width=2pt] {};
\draw (3.5,.5) node[cross=5pt,line width=2pt] {};
\draw (4,0) node[cross=5pt,line width=2pt] {};
\draw [dashed] (2.25,-.1)--(2.25,2.1);
\end{tikzpicture}
\caption{Case $m=5,$ $n=9;$ $\sigma = \{(2,6),(2,7),(4,8),(5,9)\}$ is an example of an element of C.}
\label{Element of C}
\end{center}
\end{figure}
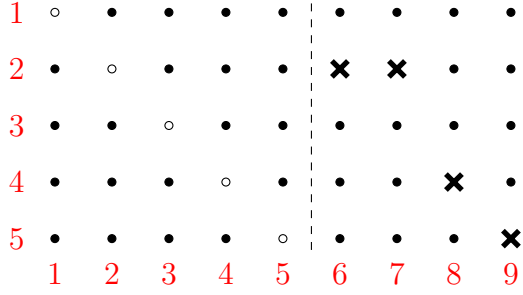

Let $G \coloneqq \{\sigma \in \Delta : H_{n-2}(\Delta,\cost_\Delta(\sigma);\Bbbk) \neq 0 \wedge \forall_{\varepsilon \subsetneq \sigma}(H_{n-2}(\Delta,\cost_\Delta(\varepsilon);\Bbbk) = 0)\}$ the set of the faces corresponding to the minimal generators of $\omega_{\Bbbk[\Delta]}.$ We want to prove the following theorem.

\begin{theorem}\label{Main}
For all $\sigma \in G,$ it holds $|\sigma|=n-m.$ In particular, $\Delta$ is level.
\end{theorem}
\begin{proof}
Notice that, for all $\sigma \in C$ of Lemma \ref{Lemma 2}, we have that $|\sigma|=n-m;$  Our aim is to prove that indeed $G=C.$\\
Let be $G_i \coloneqq \{\sigma \in G : |\sigma|=i\}$ and $C_i \coloneqq \{\sigma \in C : |\sigma|=i\}.$ We know that $C_{n-m}=C,$ then $C_i=\emptyset$ if $i \neq n-m.$ Notice that, if $|\sigma|<n-m,$ then $\sigma$ is contained in some $\tau \in$ Bd, thus this implies that $H_{n-2}(\Delta,\cost_\Delta(\sigma);\Bbbk) = 0,$ then $\sigma \notin G:$ therefore, $G_i=\emptyset$ if $i<n-m.$\\
Now let us consider $\sigma \in G_{n-m}:$ since $H_{n-2}(\Delta,\cost_\Delta(\sigma);\Bbbk) \neq 0,$ $\sigma$ cannot be contained in any $\tau \in$ Bd and, if $\varepsilon \subsetneq \sigma,$ we have just seen that $\varepsilon$ must be contained in some $\tau \in$ Bd, therefore $\sigma \in C.$ Thus we have proved that $G_{n-m} \subseteq C.$\\
By Theorem \ref{Grabe canonical} and the previous lemmas, we know that $$|G_{n-m}| = \dm_\Bbbk((\omega_{\Bbbk[\Delta]})_{n-m}) = \binom{n-1}{m-1} = |C|,$$
then $G_{n-m}=C$ and $C = G_{n-m} \subseteq G.$\\
Now suppose that $G_i \neq \emptyset$ for some $i>n-m$ and let $\sigma \in G_i,$ then we have $H_{n-2}(\Delta,\cost_\Delta(\sigma);\Bbbk) \neq 0.$ Thus, for all $\tau \in$ Bd, it holds $\sigma \nsubseteq \tau,$ i.e. $\sigma$ occupies all the last $n-m$ columns of $M.$ Notice that we can find $\varepsilon \subsetneq \sigma$ such that $\varepsilon \in C$ (we just leave one vertex for each one of the last $n-m$ columns), therefore $\varepsilon \in G.$ We have found a contraddiction, since there cannot a be proper inclusion $\varepsilon \subsetneq \sigma$ between two elements of $G.$\\
Then, for all $i>n-m,$ it holds $G_i = \emptyset,$ which implies $G=C.$ Thus we have proved the levelness of $\Delta.$\\
\end{proof}

The latter proof helps us to relate the module structure of $\omega_{\Bbbk[\Delta]}$ with the combinatorial properties of $\Delta.$ This link allows us to give a complete description of $\omega_{\Bbbk[\Delta]}$ as $\Bbbk$-vector space.

\begin{corollary}
Let us consider $G' \coloneqq \{\sigma \in \Delta : H_{n-2}(\Delta,\cost_\Delta(\sigma);\Bbbk) \neq 0\}$ and $C' \coloneqq \{\varepsilon \in \Delta : \nexists_{\sigma \in \textup{Bd}}(\varepsilon \subseteq \sigma)\}.$ Then $C'=G'.$
\end{corollary}

\begin{proof}
Notice that the elements of $G$ and $C$ are respectively the minimal elements of $G'$ and $C'$ with respect to the inclusion. Let $\varepsilon \subseteq \sigma \in \Delta:$ notice that if $\varepsilon \in G',$ then $\sigma \in G'$ and if $\varepsilon \in C',$ then $\sigma \in C'.$ Since $G=C,$ we have $G'=C'$.\\
\end{proof}

Since $\Bbbk[\Delta]$ is Cohen-Macaulay and since the Hilbert series of $S/I_2$ is well known (see \cite[Proposition 4.1.9(c)]{BruCoRaVa}), one can calculate the Hilbert series of $\omega_{\Bbbk[\Delta]}$ using \cite[Corollary 4.4.6]{BruHe93}. Anyway, in the next remark we calculate the Hilbert function of $\omega_{\Bbbk[\Delta]}$ using the combinatorial description of $\Delta.$

\begin{remark}
Using the fact that $G'=C'$, we can count the elements of $G'$. Let $G'_i \coloneqq \{\sigma \in G' : |\sigma|=i\}.$ Then, for $i \in \ZZ,$
$$|G'_i| = \sum_c \sum_r \binom{m}{r}\binom{m-r}{c}\frac{(i-1)!}{(r+c+n-m-i-1)!(m+i-n-c)!(i-r)!},$$
with $c \in \{0,\dots,m+i-n\}$ and $r \in \{m+i+1-n-c,\dots,\min(m-c,i)\}.$\\
Considering $\omega_{\Bbbk[\Delta]}$ as a $\NN$-graded module, we want to count the dimension of its $i$-th graded component. Notice that the faces in $G'_i$ gives arise to the square-free monomials of $\omega_{\Bbbk[\Delta]},$ but there are also other monomials $x_U$ which come from non-square-free vectors $U \in \NN^V$, with $|U|=i$ and $s(U) \in \Delta$. We have, for $i \in \ZZ$,
$$\dm_\Bbbk((\omega_{\Bbbk[\Delta]})_i) = \sum_{k=n-m}^{\min(i,n-1)}k^{i-k}|G'_k|.$$
\end{remark}

Here it is an example.

\begin{example}[$m=3, n=4$]\label{Delta 3,4}
In this case,
\begin{align*}
\Delta = \langle &\{(3,1),(3,2),(3,4)\}, \{(2,1),(3,1),(3,4)\}, \{(2,1),(2,4),(3,4)\},\\
&\{(2,1),(2,3),(2,4)\}, \{(1,2),(3,2),(3,4)\}, \{(1,2),(1,4),(3,4)\},\\ &\{(1,4),(2,4),(3,4)\}, \{(1,3),(2,3),(2,4)\}, \{(1,3),(1,4),(2,4)\},\\ &\{(1,2),(1,3),(1,4)\} \rangle.
\end{align*}

\begin{figure}[h]
\begin{center}
\begin{tikzpicture}[scale=1.5]
\node [black,left,red] at (0,0) {$(2,1)$};
\node [black,below left,red] at (1,-1.73) {$(3,1)$};
\node [black,below right,red] at (3,-1.73) {$(3,2)$};
\node [black,right,red] at (4,0) {$(1,2)$};
\node [black,above right,red] at (3,1.73) {$(1,3)$};
\node [black,above left,red] at (1,1.73) {$(2,3)$};
\draw [fill=blue!20] (0,0)--(1,-1.73)--(3,-1.73)--(4,0)--(3,1.73)--(1,1.73)--cycle;
\node [black,above left,red] at (3,.73) {$(1,4)$};
\node [black,left,red] at (1.5,.8) {$(2,4)$};
\node [black,below,red] at (2,-.8) {$(3,4)$};
\draw (4,0)--(2,-.5)--(3,-1.73);
\draw (1,-1.73)--(2,-.5)--(0,0);
\draw (0,0)--(1.5,.7)--(1,1.73);
\draw (1.5,.7)--(3,1.73)--(3,.73)--(4,0);
\draw (2,-.5)--(1.5,.7)--(3,.73)--(2,-.5);
\draw [fill=black] (0,0) circle (2pt);
\draw [fill=black] (1,-1.73) circle (2pt);
\draw [fill=black] (3,-1.73) circle (2pt);
\draw [fill=black] (4,0) circle (2pt);
\draw [fill=black] (3,1.73) circle (2pt);
\draw [fill=black] (1,1.73) circle (2pt);
\draw [fill=black] (1.5,.7) circle (2pt);
\draw [fill=black] (2,-.5) circle (2pt);
\draw [fill=black] (3,.73) circle (2pt);
\end{tikzpicture}
\caption{Case $m=3,$ $n=4;$ a geometric representation of $\Delta.$}
\end{center}
\end{figure}
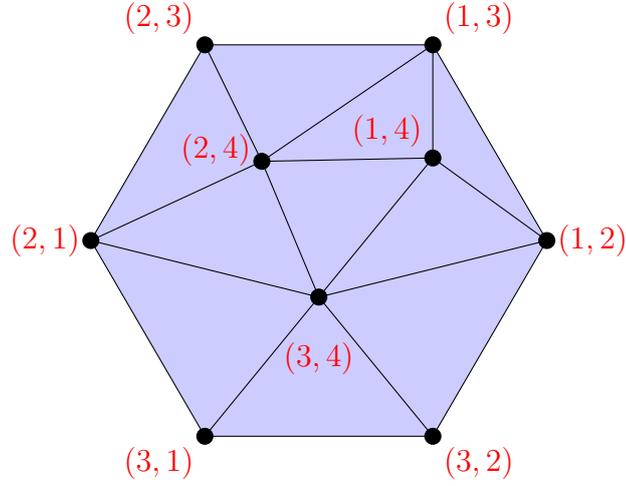

The faces correspondig to the minimal generators of $\omega_{\Bbbk[\Delta]}$ are precisely the element of $C=\{(1,4),(2,4),(3,4)\}$ and so the square-free generators are the faces that contains at least one element of $C,$ i.e. the elements of $C'.$ Let us apply the formulas of the previous remark.
\begin{itemize}
\item We know that $C'_i=\emptyset$ if $i<1$ and that $|C'_1|=|G'_1|=|C|=3.$
\item $|C'_2| = \sum_{c=0}^{1} \sum_{r=2-c}^{\min(3-c,2)} \binom{3}{r} \binom{3-r}{c} \frac{1!}{(r+c-2)!(1-c)!(2-r)!} = \dots = 12.$
\item For $C'_3$ we do not need the formula, since we know that all the facets are in $C',$ so $|C'_3|=|\FF(\Delta)|=10.$
\end{itemize}

\begin{figure}[h]
\begin{center}
\begin{tikzpicture}[scale=.8]
\node [black,left,red] at (0,0) {};
\node [black,below left,red] at (1,-1.73) {};
\node [black,below right,red] at (3,-1.73) {};
\node [black,right,red] at (4,0) {};
\node [black,above right,red] at (3,1.73) {};
\node [black,above left,red] at (1,1.73) {};
\draw [fill=blue!20] (0,0)--(1,-1.73)--(3,-1.73)--(4,0)--(3,1.73)--(1,1.73)--cycle;
\node [black,above left,red] at (3,.73) {};
\node [black,left,red] at (1.5,.8) {};
\node [black,below,red] at (2,-.8) {};
\draw (4,0)--(2,-.5)--(3,-1.73);
\draw (1,-1.73)--(2,-.5)--(0,0);
\draw (0,0)--(1.5,.7)--(1,1.73);
\draw (1.5,.7)--(3,1.73)--(3,.73)--(4,0);
\draw (2,-.5)--(1.5,.7)--(3,.73)--(2,-.5);
\draw [fill=black] (0,0) circle (2pt);
\draw [fill=black] (1,-1.73) circle (2pt);
\draw [fill=black] (3,-1.73) circle (2pt);
\draw [fill=black] (4,0) circle (2pt);
\draw [fill=black] (3,1.73) circle (2pt);
\draw [fill=black] (1,1.73) circle (2pt);
\draw [fill=red] (1.5,.7) circle (2pt);
\draw [fill=red] (2,-.5) circle (2pt);
\draw [fill=red] (3,.73) circle (2pt);
\node [black,below,red] at (2,-2) {$C'_1$};

\node [black,left,red] at (5,0) {};
\node [black,below left,red] at (6,-1.73) {};
\node [black,below right,red] at (8,-1.73) {};
\node [black,right,red] at (9,0) {};
\node [black,above right,red] at (8,1.73) {};
\node [black,above left,red] at (6,1.73) {};
\draw [fill=blue!20] (5,0)--(6,-1.73)--(8,-1.73)--(9,0)--(8,1.73)--(6,1.73)--cycle;
\node [black,above left,red] at (8,.73) {};
\node [black,left,red] at (6.5,.8) {};
\node [black,below,red] at (7,-.8) {};
\draw [red] (9,0)--(7,-.5)--(8,-1.73);
\draw [red] (6,-1.73)--(7,-.5)--(5,0);
\draw [red] (5,0)--(6.5,.7)--(6,1.73);
\draw [red] (6.5,.7)--(8,1.73)--(8,.73)--(9,0);
\draw [red] (7,-.5)--(6.5,.7)--(8,.73)--(7,-.5);
\draw [fill=black] (5,0) circle (2pt);
\draw [fill=black] (6,-1.73) circle (2pt);
\draw [fill=black] (8,-1.73) circle (2pt);
\draw [fill=black] (9,0) circle (2pt);
\draw [fill=black] (8,1.73) circle (2pt);
\draw [fill=black] (6,1.73) circle (2pt);
\draw [fill=black] (6.5,.7) circle (2pt);
\draw [fill=black] (7,-.5) circle (2pt);
\draw [fill=black] (8,.73) circle (2pt);
\node [black,below,red] at (7,-2) {$C'_2$};

\node [black,left,red] at (10,0) {};
\node [black,below left,red] at (11,-1.73) {};
\node [black,below right,red] at (13,-1.73) {};
\node [black,right,red] at (14,0) {};
\node [black,above right,red] at (13,1.73) {};
\node [black,above left,red] at (11,1.73) {};
\draw [fill=red!50] (10,0)--(11,-1.73)--(13,-1.73)--(14,0)--(13,1.73)--(11,1.73)--cycle;
\node [black,above left,red] at (13,.73) {};
\node [black,left,red] at (11.5,.8) {};
\node [black,below,red] at (12,-.8) {};
\draw (14,0)--(12,-.5)--(13,-1.73);
\draw (11,-1.73)--(12,-.5)--(10,0);
\draw (10,0)--(11.5,.7)--(11,1.73);
\draw (11.5,.7)--(13,1.73)--(13,.73)--(14,0);
\draw (12,-.5)--(11.5,.7)--(13,.73)--(12,-.5);
\draw [fill=black] (10,0) circle (2pt);
\draw [fill=black] (11,-1.73) circle (2pt);
\draw [fill=black] (13,-1.73) circle (2pt);
\draw [fill=black] (14,0) circle (2pt);
\draw [fill=black] (13,1.73) circle (2pt);
\draw [fill=black] (11,1.73) circle (2pt);
\draw [fill=black] (11.5,.7) circle (2pt);
\draw [fill=black] (12,-.5) circle (2pt);
\draw [fill=black] (13,.73) circle (2pt);
\node [black,below,red] at (12,-2) {$C'_3$};
\end{tikzpicture}
\caption{Case $m=3,$ $n=4;$ in red the elements of $C'_i.$}
\end{center}
\end{figure}
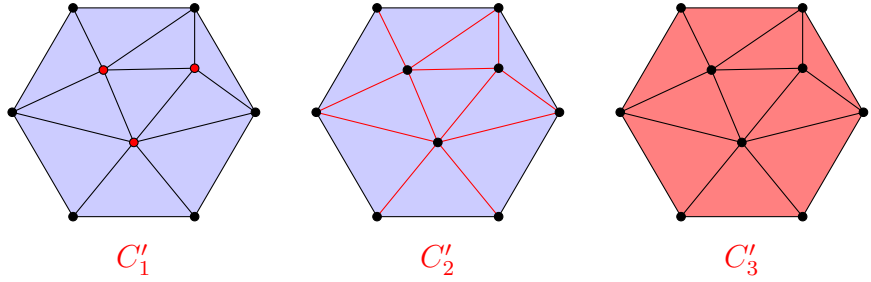

Now we calculate the $\dm_\Bbbk((\omega_{\Bbbk[\Delta]})_i)$ for some $i.$
\begin{itemize}
\item For $i<1,$ $\dm_\Bbbk((\omega_{\Bbbk[\Delta]})_i)=0.$
\item $\dm_\Bbbk((\omega_{\Bbbk[\Delta]})_1)=3,$ since in $(\omega_{\Bbbk[\Delta]})_1$ there are only square-free generators.
\item $\dm_\Bbbk((\omega_{\Bbbk[\Delta]})_2)= \sum_{k=1}^{\min(2,3)}k^{2-k}|C'_k| = |C'_1| + |C'_2| = 15.$
\item $\dm_\Bbbk((\omega_{\Bbbk[\Delta]})_3)= \sum_{k=1}^{\min(3,3)}k^{3-k}|C'_k| = |C'_1| + 2|C'_2| + |C'_3| = 37.$
\item $\dm_\Bbbk((\omega_{\Bbbk[\Delta]})_4)= \sum_{k=1}^{\min(4,3)}k^{4-k}|C'_k| = |C'_1| + 4|C'_2| + 3|C'_3| = 81.$
\end{itemize}
\end{example}

\begin{fact}\label{Betti useful}
Before the last remarks of this section, let us recall some general well-known properties about Betti tables. Let $D$ be an homogeneous ideal in a polynomial ring $T=\Bbbk[x_1,\dots,x_N]$ positively graded and fix $<$ a monomial order on $T$.
\begin{enumerate}[(a)]
\item For all $i,j$ we have $\beta_{i,j}(T/D) \leq \beta_{i,j}(T/\textup{in}_<(D))$ (see \cite[Corollary 1.6.9(a)]{BruCoRaVa}).
\item If $M$ is a finite graded $T$-module, then
$$\textup{HS}_{M}(t) = \frac{\sum_{i,j \in \ZZ}(-1)^i \beta_{i,j}(M) t^j}{(1-t)^n}.$$
In particular, since $\textup{HS}_{T/D}(t) = \textup{HS}_{T/\textup{in}_<(D)}(t)$ (see \cite[Proposition 1.6.2.(c)]{BruCoRaVa}), the alternating sum on a fixed anti-diagonal of the Betti table has the same value for $T/D$ and $T/\textup{in}_<(D)$.
\item For all $k>0$ we have $\beta_{k,k}(T/D)=\beta_{0,k}(T/D)=0.$
\item For $i \in \{0,\dots,\textup{pdim}_T(T/D)\}$ consider $t_i(T/D) = \max\{j : \beta_{i,j}(T/D) \neq 0\}$. Then, for all $i \leq N - \textup{dim}(T/D)= \textup{ht}(D)$, it holds $t_i(T/D)>t_{i-1}(T/D)$.
\end{enumerate}
\end{fact}

\begin{remark}\label{other orders}
The levelness of $\Bbbk[\Delta]$ depends strongly on the choice of the order. We are going to see, more in general, how the Betti table of $S/\textup{in}(I_2)$ changes as we change order, in particular we analyze the case $m=4$, $n=5$. We used the software Macaulay2 (\cite{M2}) to calculate Betti numbers in the case $\Bbbk=\QQ$.\\

\begin{figure}[h]
	\begin{tabular}{|l|ccccccccccccc|}
		\hline
		& 0 & 1 & 2 & 3 & 4 & 5 & 6 & 7 & 8 & 9 & 10 & 11 & 12 \\ \hline
		0 & 1 & 0 & 0 & 0 & 0 & 0 & 0 & 0 & 0 & 0 & 0 & 0 & 0 \\
		1 & 0 & 60 & 360 & 1005 & 1452 & 1050 & 300 & 0 & 0 & 0 & 0 & 0 & 0 \\
		2 & 0 & 0 & 0 & 0 & 500 & 2280 & 4455 & 4840 & 3168 & 1200 & 200 & 0 & 0 \\
		3 & 0 & 0 & 0 & 0 & 0 & 0 & 0 & 0 & 0 & 35 & 60 & 30 & 4 \\ \hline
	\end{tabular}
	\caption{Case $m=4,$ $n=5;$ the Betti table of $S/I_2$.}
	\label{Betti minori 45}
\end{figure}

\begin{figure}[h]
	\begin{tabular}{|l|ccccccccccccc|}
		\hline
		& 0 & 1 & 2 & 3 & 4 & 5 & 6 & 7 & 8 & 9 & 10 & 11 & 12 \\ \hline
		0 & 1 & 0 & 0 & 0 & 0 & 0 & 0 & 0 & 0 & 0 & 0 & 0 & 0 \\
		1 & 0 & 60 & 360 & 1025 & 1692 & 1722 & 1103 & 446 & 112 & 16 & 1 & 0 & 0 \\
		2 & 0 & 0 & 20 & 240 & 1173 & 3092 & 4939 & 5049 & 3349 & 1395 & 325 & 30 & 0 \\
		3 & 0 & 0 & 0 & 1 & 9 & 38 & 97 & 165 & 194 & 160 & 90 & 30 & 4 \\ \hline
	\end{tabular}
	\caption{Case $m=4,$ $n=5;$ the Betti table of $S/I$.}
	\label{Betti righe 45}
\end{figure}

If we choose the most natural degrevlex order induced by the following order on the variables
$$x_{1,1}>x_{1,2}> \cdots x_{1,n} > x_{2,1} > x_{2,2} > \cdots x_{2,n} > \cdots > x_{m,n}$$
and consider $J$ the initial ideal of $I_2$ with respect to that order, $S/J$ turns out to be not level if $m>2$. For the proof of this result one can use \cite[Theorem 1.2]{Yu}: from this result of Yu, it follows that $\beta_{h,h+1}(S/I_2) \geq 2$ (with $h=\textup{ht}(I_2)$). Notice that, since all the polynomial of the universal Gr\"obner basis for $I_2$ described in \cite[Section 5.3]{BruCoRaVa} are binomials with square-free terms, any initial ideal of $I_2$ is square-free; this implies (see \cite[Theorem 2.4.3(c)]{BruCoRaVa}) that $\textup{reg}(S/I_2)= \textup{reg}(S/\textup{in}(I_2))$ for any initial ideal. Therefore $\textup{reg}(S/J)= \textup{reg}(S/I_2)=m-1$ (see \cite[Remark 4.4.13]{BruCoRaVa}) and, since $\textup{pdim}(S/J)=\textup{pdim}(S/I_2)=h$ ($J$ is square-free), the Fact \ref{Betti useful}(b) implies that $\beta_{h,h+m-1}(S/J)=\beta_{h,h+m-1}(S/I_2) \neq 0.$ Thus, $S/J$ is not level.

\begin{figure}[h]
	\begin{tabular}{|l|ccccccccccccc|}
		\hline
		& 0 & 1 & 2 & 3 & 4 & 5 & 6 & 7 & 8 & 9 & 10 & 11 & 12 \\ \hline
		0 & 1 & 0 & 0 & 0 & 0 & 0 & 0 & 0 & 0 & 0 & 0 & 0 & 0 \\
		1 & 0 & 60 & 360 & 1080 & 2058 & 2744 & 2694 & 1997 & 1120 & 464 & 134 & 24 & 2 \\
		2 & 0 & 0 & 75 & 606 & 2201 & 4734 & 6680 & 6464 & 4346 & 2006 & 609 & 110 & 9 \\
		3 & 0 & 0 & 0 & 7 & 60 & 228 & 504 & 714 & 672 & 420 & 168 & 39 & 4 \\ \hline
	\end{tabular}
	\caption{Case $m=4,$ $n=5;$ the Betti table of $S/J$.}
	\label{Betti classico 45}
\end{figure}

Anyway, our monomial order is not the only one that gives arise to a level initial ideal; one can prove, using the same techniques that we use in this paper, that $S/\textup{in}_{<'}(I_2)$ is level, with $<'$ the degrevlex order induced by the following order on the variables: $x_{i,j} <' x_{k,l}$ if one among these conditions holds\\

\vspace{2cm}

\begin{enumerate}[(i)]
\item $i-j=k-l$ and $i>k,$ \\
i.e. $x_{i,j}$ and $x_{k,l}$ are in the same diagonal of $X$ and $x_{k,l}$ is higher than $x_{i,j}$;
\item $i-j<k-l,$ $i \neq j$ and $k \neq l,$ \\
i.e. the diagonal of $x_{k,l}$ is under the diagonal of $x_{i,j}$ and both are not the main diagonal; 
\item $k \neq l$ and $i=j,$ \\
i.e. $x_{i,j}$ is in the main diagonal and $x_{k,l}$ is not.
\end{enumerate}
Basically, to go from the highest to the lowest variable we scroll the diagonals from the left to the right, but we leave in the last place the main diagonal. The simplicial complex $\Delta'$ associated with $\textup{in}_{<'}(I_2)$ is very similar to $\Delta$: the difference is the fact that the facets of $\Delta'$ are paths from the bottom-left vertex to the top-right one, since in a minor without cone points the leading term is always the diagonal.

\begin{figure}[h]
	\begin{tabular}{|l|ccccccccccccc|}
		\hline
		& 0 & 1 & 2 & 3 & 4 & 5 & 6 & 7 & 8 & 9 & 10 & 11 & 12 \\ \hline
		0 & 1 & 0 & 0 & 0 & 0 & 0 & 0 & 0 & 0 & 0 & 0 & 0 & 0 \\
		1 & 0 & 60 & 360 & 1025 & 1682 & 1686 & 1064 & 442 & 127 & 23 & 2 & 0 & 0 \\
		2 & 0 & 0 & 20 & 230 & 1137 & 3035 & 4935 & 5064 & 3356 & 1396 & 325 & 30 & 0 \\
		3 & 0 & 0 & 0 & 1 & 9 & 38 & 97 & 165 & 194 & 160 & 90 & 30 & 4 \\ \hline
	\end{tabular}
	\caption{Case $m=4,$ $n=5;$ the Betti table of $S/\textup{in}_{<'}(I_2)$.}
	\label{Betti diago 45}
\end{figure}

By computing many examples with Macaulay2 using other orders, we come up with the following question.\\

\textbf{Question.} If we choose a degrevlex order on $\Bbbk[X]$ such that $x_{i,j}>x_{k,k}$ for all $i,k \in [m],$ $j \in [n]$ and $i \neq j$, is the levelness of the initial ideal of $I_2$ guaranteed for all $m$ and all $n$?\\

Unfortunately, in this more general case we cannot use the same approach to prove the levelness, since the simplicial complex associated to the initial ideal is much more complicated: the facets do not form some sort of paths anymore.

\begin{figure}[h]
	\begin{tabular}{|l|ccccccccccccc|}
		\hline
		& 0 & 1 & 2 & 3 & 4 & 5 & 6 & 7 & 8 & 9 & 10 & 11 & 12 \\ \hline
		0 & 1 & 0 & 0 & 0 & 0 & 0 & 0 & 0 & 0 & 0 & 0 & 0 & 0 \\
		1 & 0 & 60 & 360 & 1033 & 1745 & 1869 & 1317 & 620 & 191 & 35 & 3 & 0 & 0 \\
		2 & 0 & 0 & 28 & 293 & 1319 & 3297 & 5076 & 5040 & 3239 & 1284 & 273 & 21 & 0 \\
		3 & 0 & 0 & 0 & 0 & 0 & 1 & 9 & 36 & 81 & 108 & 81 & 30 & 4 \\ \hline
	\end{tabular}
	\caption{Case $m=4,$ $n=5;$ the Betti table of $S/\textup{in}(I_2)$ w.r.t. the degrevlex order induced by $x_{3,2} > x_{2,1} > x_{2,5} > x_{4,1} > x_{1,5} > x_{3,1} > x_{2,4} > x_{3,5} > x_{2,3} > x_{1,4} > x_{4,5} > x_{1,3} > x_{4,3} > x_{4,2} > x_{3,4} > x_{1,2} > x_{1,1} > x_{2,2} > x_{3,3} > x_{4,4}$.}
	\label{Betti random A 45}
\end{figure}

\begin{figure}[h]
	\begin{tabular}{|l|ccccccccccccc|}
		\hline
		& 0 & 1 & 2 & 3 & 4 & 5 & 6 & 7 & 8 & 9 & 10 & 11 & 12 \\ \hline
		0 & 1 & 0 & 0 & 0 & 0 & 0 & 0 & 0 & 0 & 0 & 0 & 0 & 0 \\
		1 & 0 & 60 & 360 & 1023 & 1683 & 1716 & 1107 & 454 & 117 & 17 & 1 & 0 & 0 \\
		2 & 0 & 0 & 18 & 231 & 1166 & 3088 & 4917 & 4987 & 3254 & 1309 & 284 & 23 & 0 \\
		3 & 0 & 0 & 0 & 0 & 1 & 8 & 30 & 69 & 108 & 119 & 83 & 30 & 4 \\ \hline
	\end{tabular}
	\caption{Case $m=4,$ $n=5;$ the Betti table of $S/\textup{in}(I_2)$ w.r.t. the degrevlex order induced by $x_{2,5} > x_{1,2} > x_{2,4} > x_{4,3} > x_{4,2} > x_{2,3} > x_{1,5} > x_{4,1} > x_{3,2} > x_{2,1} > x_{1,4} > x_{3,1} > x_{4,5} > x_{3,5} > x_{1,3} > x_{3,4} > x_{1,1} > x_{2,2} > x_{3,3} > x_{4,4}$.}
	\label{Betti random B 45}
\end{figure}
\end{remark}

\vspace{2cm}

\begin{remark}\label{Equal Betti tables}
By comparing Figure \ref{Betti minori 45} with Figure \ref{Betti righe 45}, \ref{Betti classico 45}, \ref{Betti diago 45}, \ref{Betti random A 45} and \ref{Betti random B 45}, we can observe that for certain orders some Betti numbers of $S/I_2$ and $S/\textup{in}(I_2)$ coincide. Now we are going to see for which $m,n$ there exists a monomial order $<$ on $S$ such that the Betti tables of $S/I_2$ and $S/\textup{in}_<(I_2)$ are the same.\\

The only case in which the answer was not known was for $m=3$ and $n \geq 4$. In order to solve this case we used computation.\\

In general, Betti numbers $\beta_{i,j}(M)$ of an $S$-module $M$ depend on the characteristic of $\Bbbk$. In this regard, from the universal coefficient theorem and the Hochster's formula (\cite[Theorem 5.5.1]{BruHe93}) it follows that, if $M$ is a Stanley-Reisner ring, the Betti numbers cannot decrease if we pass from characteristic $0$ to characteristic $p$ (with $p$ any prime).\\
If $M=S/I_2$, we have that Betti numbers can depend on the field if $m \geq 5$ and otherwise they do not (see \cite[Lemma 2.2 and Theorem 2.6]{Ha}).\\

Furthermore, it is well known that, for all $n \geq m \geq 2$, $\beta_{2,4}(S/I_2)=0$ in any characteristic (see \cite{Sh}), while $\beta_{3,5}(S/I_2)=0$ in characteristic 0 (see \cite{La}). We have done some computations that, together with these facts, show that the Betti tables of $S/I_2$ and $S/\textup{in}(I_2)$ are never the same unless $m=2$ or $m=n=3$.\\
In the case $m=n$ note that, without doing any computation, whenever $\textup{in}(I_2)$ is a Gorenstein initial ideal of $I_2$, from \cite[Theorem 4]{ConKaVa16}, we have:

\begin{enumerate}
\item for $m=n \geq 4$, $\beta_{3,5}(S/\textup{in}(I_2))$ cannot be zero, so the Betti tables of $S/I_2$ and $S/\textup{in}(I_2)$ cannot agree in characteristic $0$;
\item for $m=n \geq 6$, $\beta_{2,4}(S/\textup{in}(I_2))$ cannot be zero, so the Betti tables of $S/I_2$ and $S/\textup{in}(I_2)$ cannot agree in any characteristic.
\end{enumerate}

The following argument allows us to extend $(1)$ and $(2)$ to bigger $n$ and $m$ in the rectangular case $m<n$.\\

Consider $X'= \{x_{i,j}\}_{i \in [m']}^{j \in [n']}$, with $m' \geq m$ and $n' \geq n$, and let $S'= \Bbbk[X']$ and $I_2'=I_2(X')$. Given a monomial order on $S'$ we can consider its restriction on $S$ and notice that every monomial order on $S$ can be obtained in this way, thus when we compare $\textup{in}(I_2)$ and $\textup{in}(I_2')$ we are considering the "same" monomial order. Since one can check that $S/\textup{in}(I_2)$ is a direct summand of $S'/\textup{in}(I_2')$, then also $\textup{Tor}_i^S(S/\textup{in}(I_2),\Bbbk)_j$ is a direct summand of $\textup{Tor}_i^{S'}(S'/\textup{in}(I_2'),\Bbbk)_j$ for all $i$ and $j$. Therefore, if $\beta_{i,j}(S/\textup{in}(I_2)) \neq 0,$ then $\beta_{i,j}(S'/\textup{in}(I_2')) \neq 0.$\\

It follows that the Betti tables of $S/I_2$ and $S/\textup{in}(I_2)$ cannot agree for $n \geq m \geq 4$ in characteristic $0$ and for $n \geq m \geq 6$ in any characteristic.\\

Now we analyze the two cases in which the Betti tables of $S/I_2$ and $S/\textup{in}(I_2)$ coincide for some (in the first case fo every) initial ideal.\\

For $n \geq m = 2$, the Betti tables of $S/I_2$ and $S/\textup{in}(I_2)$ are equal for every monomial order.\\
We have already observed that $\textup{reg}(S/I_2) = \textup{reg}(S/\textup{in}(I_2)) = m-1$, thus in this case the Betti tables have only $2$ rows. From (b) and (c) of Fact \ref{Betti useful} it follows that the Betti tables are equal.

\begin{figure}[h]
\begin{tabular}{|l|cccccccccc|}
	\hline
	& 0 & 1 & 2 & 3 & 4 & 5 & 6 & 7 & 8 & 9 \\ \hline
	0 & 1 & 0 & 0 & 0 & 0 & 0 & 0 & 0 & 0 & 0 \\
	1 & 0 & 45 & 240 & 630 & 1008 & 1050 & 720 & 315 & 80 & 9 \\ \hline
\end{tabular}
\caption{Case $m=2$, $n=10$; the Betti table of $S/I_2$.}
\label{Betti minori 210}
\end{figure}

For $m=n=3$, we show that the Betti tables are equal if and only if the initial ideal is level (and therefore Gorenstein, since $m=n$).\\
As we said before, we know that the regularity is $2$; thus, from (b) and (c) of Fact \ref{Betti useful}, the only numbers that could be positive for $\textup{in}(I_2)$ and $0$ for $I_2$ are $\beta_{1,3}$, $\beta_{2,4}$, $\beta_{3,5}$ and $\beta_{4,5}$. From (b) it follows that
$$\beta_{3,5}(S/\textup{in}(I_2))=0 \iff \beta_{4,5}(S/\textup{in}(I_2))=0$$
and from (d)
$$\beta_{3,5}(S/\textup{in}(I_2))=0 \implies \beta_{1,3}(S/\textup{in}(I_2))=\beta_{2,4}(S/\textup{in}(I_2))=0.$$
Summing up, we have obtained that $\textup{in}(I_2)$ is level (i.e. $\beta_{4,5}(S/\textup{in}(I_2))=0$) if and only if $\beta_{4,5}(S/\textup{in}(I_2))=\beta_{3,5}(S/\textup{in}(I_2))=\beta_{2,4}(S/\textup{in}(I_2))=\beta_{1,3}(S/\textup{in}(I_2))=0$ if and only if the Betti tables of $S/I_2$ and $S/\textup{in}(I_2)$ are equal. The last equivalence follows from (b), since in each anti-diagonal there is at most one non-zero number.

\begin{figure}[h]
\begin{tabular}{|l|ccccc|}
	\hline
	& 0 & 1 & 2 & 3 & 4 \\ \hline
	0 & 1 & 0 & 0 & 0 & 0 \\
	1 & 0 & 9 & 16 & 9 & 0 \\
	2 & 0 & 0 & 0 & 0 & 1 \\ \hline
\end{tabular}
\caption{Case $m=3,$ $n=3;$ the Betti table of $S/I_2$.}
\label{Betti minori 33}
\end{figure}

We want to show that in the remaining cases, i.e. for $m \geq 3$ and $n \geq 4$, there is no initial ideal with the same Betti tables as $I_2$. The aim of the discussion below is to fill the remaining gaps, for which we do not have a conceptual proof.\\

Using Macaulay2, with the help of the package "StatePolytope" (\cite{M2a}), we computed Betti numbers for all the possible $4488$ initial ideals in the case $m=3$ and $n=4$, with $\Bbbk=\QQ$. We found out that $\beta_{3,5}(S/\textup{in}(I_2)) \geq 3$ for all the initial ideals in characteristic $0$ (and therefore in every characteristic), on the other hand we have $\beta_{3,5}(S/I_2)=0$ for every $\Bbbk$.\\
This implies that, for $m \geq 3$ and $n \geq 4$, $\beta_{3,5}(S/\textup{in}(I_2)) \neq 0$ for all the initial ideals and, since when char$(\Bbbk)=0$ we have already noticed that $\beta_{3,5}(S/I_2)=0$ for all $m$ and $n$, the statement is true in characteristic $0$.\\
Moreover, we can conclude even in characteristic $p$ for $m=3$ and any $n$ since in this case, as we have already noticed, Betti numbers of $S/I_2$ do not depend on the characteristic.\\
Since $\beta_{3,5}(S/I_2)=0$ does not hold for all $m$ and $n$ in every characteristic, we cannot use the same argument in characteristic $p$, therefore we concentrate on $\beta_{2,4}$. We have already noticed that $\beta_{2,4}(S/I_2)=0$ for all $m$ and $n$ and for every characteristic, but for $m=3$ and $n=4$ there exist some initial ideal in$(I_2)$ such that $\beta_{2,4}(S/\textup{in}(I_2))=0$.\\
Thus, it is enough to prove that, for all the initial ideals, $\beta_{2,4}(S/\textup{in}(I_2)) \neq 0$ in the case $m=4$ and $n=4$. This is been done in \cite[Example 1.4]{CoHosTh} by computing $\beta_{2,4}(S/\textup{in}(I_2))$ for all the possible initial ideals. Thus, the statement is true also in characteristic $p$.

\begin{figure}[h]
	\begin{tabular}{|l|ccccccc|}
		\hline
		& 0 & 1 & 2 & 3 & 4 & 5 & 6 \\ \hline
		0 & 1 & 0 & 0 & 0 & 0 & 0 & 0 \\
		1 & 0 & 18 & 52 & 60 & 24 & 0 & 0 \\
		2 & 0 & 0 & 0 & 0 & 10 & 12 & 3 \\ \hline
	\end{tabular}
	\caption{Case $m=3,$ $n=4;$ the Betti table of $S/I_2$.}
	\label{Betti minori 34}
\end{figure}

\begin{figure}[h]
	\begin{tabular}{|l|cccccccccc|}
		\hline
		& 0 & 1 & 2 & 3 & 4 & 5 & 6 & 7 & 8 & 9 \\ \hline
		0 & 1 & 0 & 0 & 0 & 0 & 0 & 0 & 0 & 0 & 0 \\
		1 & 0 & 36 & 160 & 315 & 288 & 100 & 0 & 0 & 0 & 0 \\
		2 & 0 & 0 & 0 & 0 & 100 & 288 & 315 & 160 & 36 & 0 \\
		3 & 0 & 0 & 0 & 0 & 0 & 0 & 0 & 0 & 0 & 1 \\ \hline
	\end{tabular}
	\caption{Case $m=4,$ $n=4;$ the Betti table of $S/I_2$.}
	\label{Betti minori 44}
\end{figure}

\end{remark}

\section{Shellability} \label{shellablility}
The shellability of a simplicial complex is an important combinatorial property, which is stronger than the Cohen-Macaulayness and so it is often used in order to prove the latter. In this section we prove the shellability of $\Delta.$ In \cite[Section 4]{CoHosTh} they prove that the simplicial complex $\Delta'$ (the one we have introduced in the Remark \ref{other orders}) is shellable in the case $m=n$. One can apply the same argument, with proper modifications, to prove the shellability of $\Delta$ in the case $m=n.$ In this section we generalize their proof to the case $m<n.$

\begin{definition}
A pure $d$-dimensional simplicial complex $\Gamma$ is said to be \textbf{shellable} if there exists a total order on its facets $F_1,\dots,F_c$ (which is called a \textbf{shelling order}) such that, for all $i \in \{2,\dots,c\},$ $\langle F_i \rangle \cap \langle F_1,\dots,F_{i-1} \rangle$ is a pure ($d-1$)-dimensional simplicial complex. Moreover, the order is a \textbf{two-way shelling} if also $F_c,\dots,F_1$ is shelling.
\end{definition}

We have already noticed that if we fix $F$ a facet of $\Delta$ there exists a unique minimal submatrix $M^F$ that contains $F.$ Thus the set of the facets of $\Delta$ can be written as the following disjoint union
$$\FF(\Delta) = \displaystyle \bigsqcup_{\emptyset \subsetneq r \subseteq [m]} \{F : r^F=r\}.$$
Therefore, to provide a total order on $\FF(\Delta),$ we first order all the submatrices of the type $M^F$ for some $F \in \FF(\Delta)$ and then order all the facets in each submatrix. Notice that every submatrix of this type is uniquely determined by the set of its rows $r^F.$

\begin{definition}
Let $\emptyset \subsetneq r_1, r_2 \subseteq [m],$ $r_1 \neq r_2.$ Then we can write $r_1=\{a_1,\dots,a_c\}$ and $r_2=\{b_1,\dots,b_d\},$ with $a_1<\dots<a_c$ and $b_1<\dots<b_d.$ We say that
$$r_1 < r_2 \iff \begin{cases}
	a_i<b_i \text{ for the smallest $i$ such that $a_i \neq b_i$} \\
	\hspace{3cm} \text{or} \\
	c<d \text{ and $a_i=b_i$ for all $i \in [c]$}
\end{cases}.$$
Now let us consider $F,G \in \FF(\Delta)$ such that $r^F=r^G$ and $F \neq G.$ We say that $F<G$ if the first vertex in which they differ, starting from the top-left one, is one step to the right (respect to the previous one) in $F$ and one step lower in $G.$
Finally, let $F,G$ be two different facets.
$$F<G \iff \begin{cases}
	r^F < r^G \\
	\hspace{.5cm} \text{or} \\
	r^F = r^G \text{ and } F<G \text{ in the sense that we have just said}
\end{cases}.$$
\end{definition}

\begin{definition}
Let $F \in \FF(\Delta).$ We define the following three subsets of $F$ which form a partition of $F:$
\begin{align*}
F^+ \coloneqq & \{v \in F : v \text{ is right turning}\} \cup \\
& \{v \in F : v=(i,j) \text{ is horizontal and } \max(r^F) < j \leq m\} \cup \\
& \{v \in F : v=(i,j) \text{ is vertical and } i < \max(r^F)\}, \\
\\
F^- \coloneqq & \{v \in F : v \text{ is left turning}\} \cup \\
& \{v \in F : v=(i,j) \text{ is horizontal and } j < \max(r^F)\} \cup \\
& \{v \in F : v=(i,j) \text{ is vertical and } i = \max(r^F)\}, \\
\\
F^* \coloneqq & \{v \in F : v=(i,j) \text{ is horizontal and } j > m\}.
\end{align*}
\end{definition}

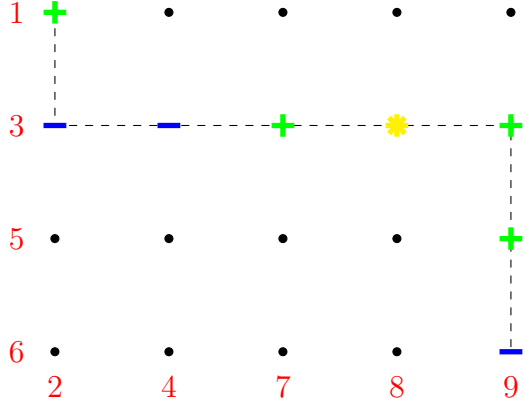
\begin{figure}[h]
\begin{center}
\begin{tikzpicture}[scale=1.5]
\node [black,above,text=red] at (0,-.5) {$2$};
\node [black,above,text=red] at (1,-.5) {$4$};
\node [black,above,text=red] at (2,-.5) {$7$};
\node [black,above,text=red] at (3,-.5) {$8$};
\node [black,above,text=red] at (4,-.5) {$9$};
\node [black,right,text=red] at (-.5,0) {$6$};
\node [black,right,text=red] at (-.5,1) {$5$};
\node [black,right,text=red] at (-.5,2) {$3$};
\node [black,right,text=red] at (-.5,3) {$1$};
\draw [dashed] (0,3)--(0,2)--(1,2)--(2,2)--(3,2)--(4,2)--(4,1)--(4,0);
\draw (0,3) node[cross=5pt,rotate=45,line width=2pt,green] {};
\draw [line width=2pt,blue] (-.1,2)--(.1,2);
\draw [line width=2pt,blue] (.9,2)--(1.1,2);
\draw (2,2) node[cross=5pt,rotate=45,line width=2pt,green] {};
\draw (3,2) node[cross=5pt,rotate=45,line width=2pt,yellow] {};
\draw (3,2) node[cross=5pt,line width=2pt,yellow] {};
\draw (4,2) node[cross=5pt,rotate=45,line width=2pt,green] {};
\draw (4,1) node[cross=5pt,rotate=45,line width=2pt,green] {};
\draw [line width=2pt,blue] (3.9,0)--(4.1,0);
\draw [fill=black] (0,0) circle (1pt);
\draw [fill=black] (1,0) circle (1pt);
\draw [fill=black] (2,0) circle (1pt);
\draw [fill=black] (3,0) circle (1pt);
\draw [fill=black] (0,1) circle (1pt);
\draw [fill=black] (1,1) circle (1pt);
\draw [fill=black] (2,1) circle (1pt);
\draw [fill=black] (3,1) circle (1pt);
\draw [fill=black] (1,3) circle (1pt);
\draw [fill=black] (2,3) circle (1pt);
\draw [fill=black] (3,3) circle (1pt);
\draw [fill=black] (4,3) circle (1pt);
\end{tikzpicture}
\caption{Case $m=7,$ $n=9;$ $r^F=\{1,3,5,6\},$\\ $F^+=\{(1,2),(3,7),(3,9),(5,9)\},$ $F^-=\{(3,2),(3,4),(6,9)\},$\\ $F^*=\{(3,8)\}.$}
\end{center}
\end{figure}

\begin{theorem}
For every facet $F \in \FF(\Delta)$ we have:
\begin{enumerate}[(i)]
\item $\langle F \rangle \cap \langle G \in \FF(\Delta) : G < F \rangle = \langle F \smallsetminus \{v\} : v \in F^- \rangle;$
\item $\langle F \rangle \cap \langle G \in \FF(\Delta) : G > F \rangle = \langle F \smallsetminus \{v\} : v \in F^+ \rangle.$
\end{enumerate}
Notice that in $(i)$ we have $\emptyset = \emptyset$ if $F$ is the minimal facet and in $(ii)$ we have $\emptyset = \emptyset$ if $F$ is the maximal facet.\\
In particular the order $<$ on the facets is a two-way shelling of $\Delta.$
\end{theorem}

\begin{proof}
\textbf{Step 1.} First we prove that for every facet $F$ and for $v=(i,j) \in F$ we have one of the following three cases:
\begin{enumerate}[(a)]
\item $v \in F^* \iff$ the only facet that contains $F \smallsetminus \{v\}$ is $F,$
\item $v \in F^+ \iff$ there exists another facet $G \supset F \smallsetminus \{v\}$ such that $G>F,$
\item $v \in F^- \iff$ there exists another facet $G \supset F \smallsetminus \{v\}$ such that $G<F.$
\end{enumerate}
Notice that is suffices to prove the directions $"\hspace{-.25cm} \implies \hspace{-.25cm}"$ since the three conditions on the left are mutually exclusive and so are the ones on the right. $(\text{a})$ is clear from Remark \ref{Boundary faces}. If $v \notin F^*,$ there exists another facet $G$ such that $G \neq F$ and $F \smallsetminus \{v\} \subset G.$\\
Let $v \in F^+.$ If $v$ is a right turning vertex, then $G$ is obtained by flipping $v$ and of course $G>F.$ If $v$ is an horizontal vertex, then $r^G = r^F \cup \{j\}$ and, since $j > \max(r^F),$ we have $r^G > r^F$ (and then $G>F$). If $v$ is a verical vertex, then $r^G = r^F \smallsetminus \{i\}$ and, since $i < \max(r^F),$ we have again $G>F.$\\
Let $v \in F^-.$ If $v$ is a left turning vertex, then $G$ is obtained by flipping $v$ and this time $G<F.$ If $v$ is an horizontal vertex, then $r^G = r^F \cup \{j\}$ and, since $j < \max(r^F),$ we have $G<F.$ If $v$ is a verical vertex, then $r^G = r^F \smallsetminus \{i\}$ and, since $i = \max(r^F),$ we have $G<F.$\\

\textbf{Step 2.} We prove that, for every pair of facets $G<F,$ there exists $v \in F^-$ such that $v \notin G.$\\
If $M^G=M^F,$ then at least one of the left turning vertices of $F$ must not be in $G,$ since if we consider the set of its left turning vertices, $F$ is the smallest facet in $M^F$ that contains this set.\\
If $M^G \neq M^F,$ then let $r^G = \{a_1,\dots,a_c\}$ and $r^F = \{b_1,\dots,b_d\}$ with $a_1 < \dots < a_c$ and $b_1 < \dots < b_d.$ Since $G<F,$ we must have that $r^G < r^F.$\\
If $a_i<b_i$ for some $i$ and $a_j=b_j$ for all $j<i,$ then the vertices in the column $a_i$ are not in $G,$ but one of those vertices is in $F.$ It suffices to find a vertex in this column that belongs to $F^-.$ If $F$ has only one vertex in the column $a_i$ we are done since $a_i<b_i \leq \max(r^F).$ If $F$ has a left turning vertex in the column $a_i$ we are done. Otherwise $a_i$ is the last column of $M^F$ and that column has at least two vertices; then $v=(\max(r^F),a_i)$, the last vertex of $F$, is what we are looking for.\\
If $c<d$ and $a_i=b_i$ for all $i \in [c],$ then the vertices in the row $b_d$ are not in $G$ but some vertex in the row $b_d$ is in $F$ and it sufficies to show that one vertex in this row belongs to $F^-.$ Since $d>1,$ $b_d$ is the last of at least two rows of $G$ and in this row we have either a vertical vertex or a left turning vertex.\\

\textbf{Step 3.} We prove that for every pair of facets $G>F,$ there exists $v \in F^+$ such that $v \notin G.$\\
Similarly to the latter step, if $M^F=M^G$ we are done by taking a right turning vertex of $F$ which is not in $G$. If $M^F \neq M^G$, we consider $r^F = \{a_1,\dots,a_c\}$ and $r^G = \{b_1,\dots,b_d\}$, with $a_1 < \dots < a_c$ and $b_1 < \dots < b_d$ such that $r^G > r^F.$\\
If $a_i<b_i$ for some $i$ and $a_j=b_j$ for all $j<i,$ then the row $a_i$ is in $M^F$ but not in $M^G;$ so let us see what can appen in the row $a_i$ of $M^F.$ If there is a right turning vertex we are done. If $a_i<\max(r^F)$ (i.e. $i<c$) and there is a vertical vertex we are done. Thus we can assume that $i=c$ and in the row $a_c$ there is no right turning vertex. If $c=1,$ then $r^F = \{a_1\}$ and $F$ is an horizontal path, so we have that $v=(a_1,b_1) \in F^+$ (notice that $\max(r^F) = a_1 < b_1 \leq m,$). If $c>1,$ then if in the column $a_c$ there is either a vertical vertex or an horizontal one $(a_c,k)$ such that $a_c<k \leq m,$ we have that this vertex belongs to $F^+.$ Otherwise in the row $a_c$ there are no vertices of $F^+$, but we can look for some vertex in the column $b_c$ (which stays in $M^F$ but not in $M^G$). If in this column there is a right turning or an horizontal vertex we are done, otherwise $b_c$ is the first column of $M^F$ and it lacks of horizontal vertices, but in this case $v=(a_1,b_c)$ is a vertical vertex and belongs to $F^+$.\\
If $c<d$ and $a_i=b_i$ for all $i \in [c],$ then the column $b_d$ is in $M^F$ but not in $M^G$ and we want to prove that this column has some vertex in $F^+.$ If in the column $b_d$ there is an horizontal or a right turning vertex we are done. Otherwise $b_d$ is the first column in $M^F$ and the first step of $F$ is vertical (in particular $c>1$). We have already noticed that in this case $v=(a_1,b_d) \in F^+$.\\

\textbf{Step 4.} Conclusion of the proof.\\
To colclude, notice that Step 1 proves $\supseteq$ in $(i)$ and $(ii),$ Step 2 proves $\subseteq$ in $(i)$ and Step 3 proves $\subseteq$ in $(ii).$\\
\end{proof}

\begin{example}
We resume the Example \ref{Delta 3,4}. The followings are the facets of $\Delta$ in ascending order.

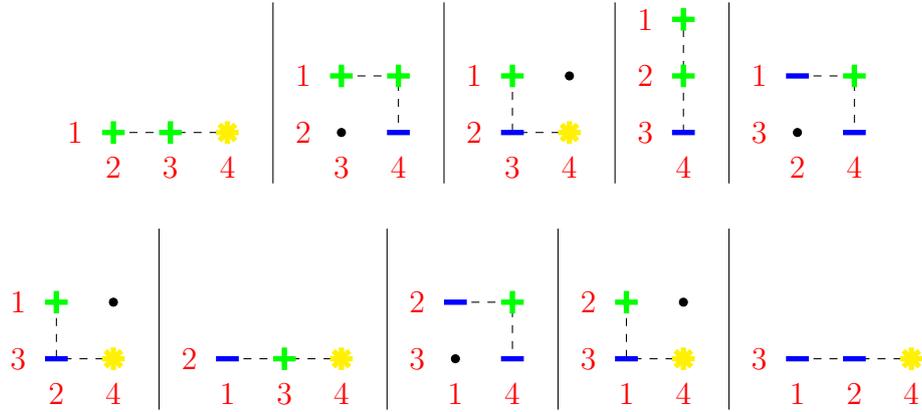
\begin{figure}[h]
\begin{center}
\begin{tikzpicture}[scale=1.5]
\node [black,above,text=red] at (0,-.5) {$2$};
\node [black,above,text=red] at (.5,-.5) {$3$};
\node [black,above,text=red] at (1,-.5) {$4$};
\node [black,right,text=red] at (-.5,0) {$1$};
\draw [dashed] (0,0)--(.5,0)--(1,0);
\draw (0,0) node[cross=5pt,rotate=45,line width=2pt,green] {};
\draw (.5,0) node[cross=5pt,rotate=45,line width=2pt,green] {};
\draw (1,0) node[cross=5pt,rotate=45,line width=2pt,yellow] {};
\draw (1,0) node[cross=5pt,line width=2pt,yellow] {};
\draw (1.4,-.45)--(1.4,1.15);

\node [black,above,text=red] at (2,-.5) {$3$};
\node [black,above,text=red] at (2.5,-.5) {$4$};
\node [black,right,text=red] at (1.5,0) {$2$};
\node [black,right,text=red] at (1.5,.5) {$1$};
\draw [dashed] (2,.5)--(2.5,.5)--(2.5,0);
\draw [line width=2pt,blue] (2.4,0)--(2.6,0);
\draw [fill=black] (2,0) circle (1pt);
\draw (2,.5) node[cross=5pt,rotate=45,line width=2pt,green] {};
\draw (2.5,.5) node[cross=5pt,rotate=45,line width=2pt,green] {};
\draw (2.9,-.45)--(2.9,1.15);

\node [black,above,text=red] at (3.5,-.5) {$3$};
\node [black,above,text=red] at (4,-.5) {$4$};
\node [black,right,text=red] at (3,0) {$2$};
\node [black,right,text=red] at (3,.5) {$1$};
\draw [dashed] (3.5,.5)--(3.5,0)--(4,0);
\draw (3.5,.5) node[cross=5pt,rotate=45,line width=2pt,green] {};
\draw [line width=2pt,blue] (3.4,0)--(3.6,0);
\draw [fill=black] (4,.5) circle (1pt);
\draw (4,0) node[cross=5pt,rotate=45,line width=2pt,yellow] {};
\draw (4,0) node[cross=5pt,line width=2pt,yellow] {};
\draw (4.4,-.45)--(4.4,1.15);

\node [black,above,text=red] at (5,-.5) {$4$};
\node [black,right,text=red] at (4.5,0) {$3$};
\node [black,right,text=red] at (4.5,.5) {$2$};
\node [black,right,text=red] at (4.5,1) {$1$};
\draw [dashed] (5,0)--(5,.5)--(5,1);
\draw [line width=2pt,blue] (4.9,0)--(5.1,0);
\draw (5,.5) node[cross=5pt,rotate=45,line width=2pt,green] {};
\draw (5,1) node[cross=5pt,rotate=45,line width=2pt,green] {};
\draw (5.4,-.45)--(5.4,1.15);

\node [black,above,text=red] at (6,-.5) {$2$};
\node [black,above,text=red] at (6.5,-.5) {$4$};
\node [black,right,text=red] at (5.5,0) {$3$};
\node [black,right,text=red] at (5.5,.5) {$1$};
\draw [dashed] (6,.5)--(6.5,.5)--(6.5,0);
\draw [line width=2pt,blue] (5.9,.5)--(6.1,.5);
\draw [line width=2pt,blue] (6.4,0)--(6.6,0);
\draw [fill=black] (6,0) circle (1pt);
\draw (6.5,.5) node[cross=5pt,rotate=45,line width=2pt,green] {};

\node [black,above,text=red] at (-.5,-2.5) {$2$};
\node [black,above,text=red] at (0,-2.5) {$4$};
\node [black,right,text=red] at (-1,-2) {$3$};
\node [black,right,text=red] at (-1,-1.5) {$1$};
\draw [dashed] (-.5,-1.5)--(-.5,-2)--(0,-2);
\draw [line width=2pt,blue] (-.6,-2)--(-.4,-2);
\draw [fill=black] (0,-1.5) circle (1pt);
\draw (-.5,-1.5) node[cross=5pt,rotate=45,line width=2pt,green] {};
\draw (0,-2) node[cross=5pt,rotate=45,line width=2pt,yellow] {};
\draw (0,-2) node[cross=5pt,line width=2pt,yellow] {};
\draw (.4,-2.45)--(.4,-.85);

\node [black,above,text=red] at (1,-2.5) {$1$};
\node [black,above,text=red] at (1.5,-2.5) {$3$};
\node [black,above,text=red] at (2,-2.5) {$4$};
\node [black,right,text=red] at (.5,-2) {$2$};
\draw [dashed] (1,-2)--(1.5,-2)--(2,-2);
\draw [line width=2pt,blue] (.9,-2)--(1.1,-2);
\draw (1.5,-2) node[cross=5pt,rotate=45,line width=2pt,green] {};
\draw (2,-2) node[cross=5pt,rotate=45,line width=2pt,yellow] {};
\draw (2,-2) node[cross=5pt,line width=2pt,yellow] {};
\draw (2.4,-2.45)--(2.4,-.85);

\node [black,above,text=red] at (3,-2.5) {$1$};
\node [black,above,text=red] at (3.5,-2.5) {$4$};
\node [black,right,text=red] at (2.5,-2) {$3$};
\node [black,right,text=red] at (2.5,-1.5) {$2$};
\draw [dashed] (3,-1.5)--(3.5,-1.5)--(3.5,-2);
\draw [line width=2pt,blue] (2.9,-1.5)--(3.1,-1.5);
\draw [line width=2pt,blue] (3.4,-2)--(3.6,-2);
\draw [fill=black] (3,-2) circle (1pt);
\draw (3.5,-1.5) node[cross=5pt,rotate=45,line width=2pt,green] {};
\draw (3.9,-2.45)--(3.9,-.85);

\node [black,above,text=red] at (4.5,-2.5) {$1$};
\node [black,above,text=red] at (5,-2.5) {$4$};
\node [black,right,text=red] at (4,-2) {$3$};
\node [black,right,text=red] at (4,-1.5) {$2$};
\draw [dashed] (4.5,-1.5)--(4.5,-2)--(5,-2);
\draw [line width=2pt,blue] (4.4,-2)--(4.6,-2);
\draw [fill=black] (5,-1.5) circle (1pt);
\draw (4.5,-1.5) node[cross=5pt,rotate=45,line width=2pt,green] {};
\draw (5,-2) node[cross=5pt,rotate=45,line width=2pt,yellow] {};
\draw (5,-2) node[cross=5pt,line width=2pt,yellow] {};
\draw (5.4,-2.45)--(5.4,-.85);

\node [black,above,text=red] at (6,-2.5) {$1$};
\node [black,above,text=red] at (6.5,-2.5) {$2$};
\node [black,above,text=red] at (7,-2.5) {$4$};
\node [black,right,text=red] at (5.5,-2) {$3$};
\draw [dashed] (6,-2)--(6.5,-2)--(7,-2);
\draw [line width=2pt,blue] (5.9,-2)--(6.1,-2);
\draw [line width=2pt,blue] (6.4,-2)--(6.6,-2);
\draw (7,-2) node[cross=5pt,rotate=45,line width=2pt,yellow] {};
\draw (7,-2) node[cross=5pt,line width=2pt,yellow] {};

\end{tikzpicture}
\caption{Case $m=3,$ $n=4;$ the facets of $\Delta.$}
\end{center}

\end{figure}

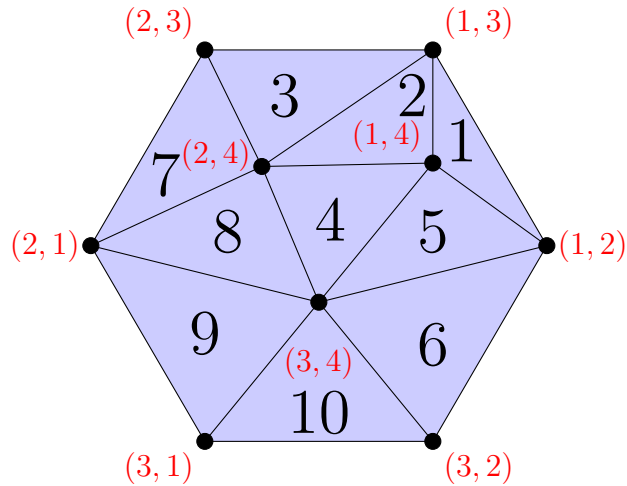
\begin{figure}[h]
\begin{center}
\begin{tikzpicture}[scale=1.5]
\node [black,left,red] at (0,0) {$(2,1)$};
\node [black,below left,red] at (1,-1.73) {$(3,1)$};
\node [black,below right,red] at (3,-1.73) {$(3,2)$};
\node [black,right,red] at (4,0) {$(1,2)$};
\node [black,above right,red] at (3,1.73) {$(1,3)$};
\node [black,above left,red] at (1,1.73) {$(2,3)$};
\draw [fill=blue!20] (0,0)--(1,-1.73)--(3,-1.73)--(4,0)--(3,1.73)--(1,1.73)--cycle;
\node [black,above left,red] at (3,.73) {$(1,4)$};
\node [black,left,red] at (1.5,.8) {$(2,4)$};
\node [black,below,red] at (2,-.8) {$(3,4)$};
\draw (4,0)--(2,-.5)--(3,-1.73);
\draw (1,-1.73)--(2,-.5)--(0,0);
\draw (0,0)--(1.5,.7)--(1,1.73);
\draw (1.5,.7)--(3,1.73)--(3,.73)--(4,0);
\draw (2,-.5)--(1.5,.7)--(3,.73)--(2,-.5);
\draw [fill=black] (0,0) circle (2pt);
\draw [fill=black] (1,-1.73) circle (2pt);
\draw [fill=black] (3,-1.73) circle (2pt);
\draw [fill=black] (4,0) circle (2pt);
\draw [fill=black] (3,1.73) circle (2pt);
\draw [fill=black] (1,1.73) circle (2pt);
\draw [fill=black] (1.5,.7) circle (2pt);
\draw [fill=black] (2,-.5) circle (2pt);
\draw [fill=black] (3,.73) circle (2pt);
\node [black,below,scale=2] at (2,-1.1) {$10$};
\node [black,below,scale=2] at (1,-.4) {$9$};
\node [black,below,scale=2] at (1.2,.5) {$8$};
\node [black,below,scale=2] at (2.1,.6) {$4$};
\node [black,below,scale=2] at (3,-.5) {$6$};
\node [black,below,scale=2] at (3,.5) {$5$};
\node [black,below,scale=2] at (.65,1) {$7$};
\node [black,below,scale=2] at (1.7,1.7) {$3$};
\node [black,below,scale=2] at (2.82,1.7) {$2$};
\node [black,below,scale=2] at (3.25,1.3) {$1$};
\end{tikzpicture}
\caption{Case $m=3,$ $n=4;$ a two-way shelling order for $\Delta.$}
\end{center}
\end{figure}

\end{example}

$\vspace{6cm}$



\bibliographystyle{acm}
\bibliography{AJME.bib}

\end{document}